\documentclass[11pt]{article}

\usepackage[english]{babel}
\usepackage{amsmath,amssymb,amsthm}
\usepackage{enumerate}
\usepackage[affil-it]{authblk}
\usepackage[pdftex]{graphicx,color}
\usepackage{color}
\usepackage{hyperref}
\usepackage{subcaption}
\usepackage{comment}
\usepackage{mathtools}

\usepackage{pgf}
\usepackage{tikz}
\usetikzlibrary{arrows,automata}
\usepackage[latin1]{inputenc}
\usepackage{verbatim}
\usepackage{authblk}

\newtheorem{theorem}{Theorem}[section]
\newtheorem{proposition}[theorem]{Proposition}
\newtheorem{lemma}[theorem]{Lemma}
\newtheorem{corollary}[theorem]{Corollary}

\theoremstyle{definition}
\newtheorem{definition}[theorem]{Definition}
\newtheorem{example}[theorem]{Example}
\newtheorem{remark}{Remark}

\newcommand{\Z}{\mathbb{Z}}
\newcommand{\N}{\mathbb{N}}
\newcommand{\R}{\mathbb{R}}

\newcommand{\mx}{\mathbf{x}}

\newcommand{\my}{\mathbf{y}}

\newcommand{\mm}{\mathbf{m}}

\newcommand{\E}{\operatorname{{\mathbb E}}}

\renewcommand{\P}{\operatorname{\mathbb{P}}}

\newcommand{\tm}{\mathrm{t_{mix}}}

\newcommand{\T}{\mathcal{T}}

\begin{document}

\title{Random Transpositions on Contingency Tables}
\author{ Mackenzie Simper \thanks{\href{mailto:msimper@stanford.edu}{\nolinkurl{msimper@stanford.edu}}}}
\affil{Stanford University}
\date{}

\maketitle

\begin{abstract}
    Contingency tables are useful objects in statistics for representing 2-way data. With fixed row and column sums, and a total of $n$ entries, contingency tables correspond to parabolic double cosets of $S_n$. The uniform distribution on $S_n$ induces the Fisher-Yates distribution, classical for its use in the chi-squared test for independence.  A Markov chain on $S_n$ can then induce a random process on the space of contingency tables through the double cosets correspondence. The random transpositions Markov chain on $S_n$ induces a natural `swap' Markov chain on the space of contingency tables; the stationary distribution of the Markov chain is the Fisher-Yates distribution. This paper describes this Markov chain and shows that the eigenfunctions are orthogonal polynomials of the Fisher-Yates distribution. Results for the mixing time are discussed, as well as connections with sampling from the uniform distribution on contingency tables, and data analysis.  
    
\end{abstract}

\section{Introduction}



Contingency tables are a classical object of study, useful in statistics for representing data with two categorical features. For example, rows could be hair colors and columns could be eye colors. The number in an entry in the table counts the number of individuals from a sample with a fixed hair, eye color combination. Various statistical tests exist to test the hypothesis that the row and column features are independent, and other more general models. See e.g.\ \cite{lancaster}, \cite{agresti92} and Section 5 of \cite{diaconisSimper} for many more references. 

Given $\lambda = (\lambda_1, \dots, \lambda_I), \mu = (\mu_1, \dots, \mu_J)$ two partitions of $n$, the set of contingency tables $\T_{\lambda, \mu}$ is the set of $I \times J$ tables with non-negative integer entries with row sums $\lambda_1, \dots, \lambda_I$ and column sums $\mu_1, \dots, \mu_J$. As explained below, the space $\T_{\lambda, \mu}$ is in bijection with the space of double cosets $S_\lambda \backslash S_n / S_\mu$, where $S_n$ is the symmetric group on $n$ elements and $S_\lambda, S_\mu$ are the Young subgroups defined by $\lambda, \mu$. Though this is a classical fact \cite{jamesKerber}, recently it was studied in \cite{diaconisSimper} with probabilistic motivation. The general question is: Given a finite group $G$ and two subgroups $H, K$, what is the distribution on $H \backslash G / K$ induced by picking a element uniformly from $G$ and mapping it to its double coset? For $S_\lambda \backslash S_n / S_\mu$, the induced distribution on $\T_{\lambda, \mu}$ is the \emph{Fisher-Yates} distribution:
\begin{equation}
    \pi_{\lambda, \mu}(T) = \frac{1}{n!}\prod_{i,j} \frac{\lambda_i ! \mu_j!}{T_{ij}!}, \quad T = \{T_{ij} \}_{1 \le i \le I, 1 \le j \le J} \in \T_{\lambda, \mu}.
\end{equation}

The Fisher-Yates distribution is classically defined as the conditional distribution on a contingency table sampled with cell probabilities $p_{ij}$ which satisfy $p_{ij} = \theta_{i} \eta_j$ (the assumption that row and column features are independent). It is the distribution of the table conditioned on the sufficient statistics (the row and column sums). It is a pleasant surprise that this same distribution is induced by the double coset decomposition.

Whenever there is a Markov chain on some state space, one can consider the lumped process on any partitioning of this space. In \cite{diaconisRamSimper} the question is studied for double cosets: Given a Markov chain on $G$, is the lumped process on $H \backslash G / K$ also a Markov chain? An affirmative answer is proven when the Markov chain on $G$ is induced by a probability measure on $G$ which is constant on conjugacy classes. 

Many Markov chains on $S_n$ have been studied, and perhaps the most famous and easiest to describe is the \emph{random transpositions} Markov chain \cite{diaconisShahshahani1981}. This is generated by the class function $Q$ which gives equal probability to transpositions. The Markov chain on $\T_{\lambda, \mu}$ induced by random transpositions turns out to be novel and easy to describe. One move involves adding a sub-matrix of the form $\begin{pmatrix} +1 & -1 \cr -1 & +1 \end{pmatrix}$ to the current table; the probability of selecting the rows/columns for the sub-matrix depend on the current configuration. Since the stationary distribution of the random transpositions chain on $S_n$ is uniform, the induced chain on $\T_{\lambda, \mu}$ has Fisher-Yates distribution as stationary distribution. These statements are proven in Section \ref{sec: CTsubIntro} below.

This same move has previously been studied for a Markov chain on contingency tables with rows and columns for the sub-matrix chosen uniformly at random (and moves which would give a negative entry in the table are rejected). This `uniform swap' chain was proposed in \cite{diaconisGangoli} to sample contingency tables from the uniform distribution, as a way of approximately counting the number of tables, which is a \# -P complete problem.

The purpose of this paper is to describe the random transposition chain on $\T_{\lambda, \mu}$, characterize its eigenvalues and eigenvectors, and analyze the mixing time in special cases. This chain has polynomial eigenvectors, which can be used to give orthogonal polynomials for the Fisher-Yates distribution. In the special case of a two-rowed table, the Fisher-Yates distribution is the multivariate hypergeometric distribution and the orthogonal polynomials are the multivariate Hahn polynomials. More generally, this Markov chain gives a characterization of the natural analog of Hahn polynomials for the Fisher-Yates distribution.


The remainder of this section reviews the correspondence between contingency tables and double cosets, defines the Markov chain explicitly, states the main results of this paper, and reviews related work.

\subsection{Contingency Tables as Double Cosets} \label{sec: CTsubIntro}

This section describes the relationship between contingency tables and double cosets of $S_n$. Let $n$ be a positive integer and suppose $\lambda = (\lambda_1, \lambda_2, \dots, \lambda_I)$ is a partition of $n$. That is, $\lambda_i$ are all positive integers, $\lambda_1 \ge \lambda_2 \ge \hdots \ge \lambda_I > 0$, and $\lambda_1 + \lambda_2 + \hdots + \lambda_I = n$. 

\begin{definition}[Young Subgroup]
The \emph{parabolic subgroup}, or \emph{Young subgroup}, $S_\lambda$ is the set of all permutations in $S_n$ which permute only $\{1, 2, \dots, \lambda_1 \}$ among themselves, only $\{\lambda_1 + 1, \dots, \lambda_1 + \lambda_2 \}$ among themselves, and so on. Thus,
\[
S_\lambda \cong S_{\lambda_1} \times S_{\lambda_2} \times \hdots \times S_{\lambda_I}.
\]
\end{definition}

If $G$ is an arbitrary finite group and $H, K$ are two sub-groups of $G$, then the $H-K$ \emph{double-cosets} are the equivalence classes defined by the equivalence relation
\[
s \sim t \iff s = h t k \,\,\,\,\,\,\,\, \text{for} \,\,\,\,\,\,\,\, s, t \in G, \,\,\, h \in H, \,\,\, k \in K.
\]

Let $\mu = (\mu_1, \dots, \mu_J)$ be a second partition of $n$, let $\T_{\lambda, \mu}$ denote the space of contingency tables with row sums given by $\lambda$ and column sums $\mu$. That is, an element of $\T_{\lambda, \mu}$ is an $I \times J$ table with non-negative integer entries, row sums $\lambda_1, \dots, \lambda_I$ and column sums $\mu_1, \dots, \mu_J$. The following bijection is described in further detail in \cite{jamesKerber}, and the induced distribution, along with many more references, in \cite{diaconisSimper}.

\begin{lemma} \label{lem: CTcorrespondence}
The set of double-cosets $S_\lambda \backslash S_n / S_\mu$ is in bijection with $\T_{\lambda, \mu}$. The distribution on $\T_{\lambda, \mu}$ induced by sampling a uniform $\sigma \in S_n$ and mapping it to it's corresponding double-coset is the \emph{Fisher-Yates} distribution
 \begin{equation} \label{eqn: fisher-yates}
    \pi_{\lambda, \mu}(T) = \frac{1}{n!} \prod_{i, j} \frac{\mu_i! \lambda_j!}{T_{ij}!}, \quad T \in \T_{\lambda, \mu}.
    \end{equation}
\end{lemma}

\begin{remark}
Without loss of generality, we assume that the rows and columns of the contingency tables $\T_{\lambda, \mu}$ are ordered, i.e.\ the first row/column has the largest sum and then row/column sums are decreasing. This ordering does not effect any of the results in this paper. 
\end{remark}

The mapping from $S_n$ to tables is easy to describe: Fix $\sigma \in S_n$. Inspect the first $\lambda_1$ positions in $\sigma$. Let $T_{11}$ be the number of elements from $\{1, 2, \dots, \mu_1 \}$ occurring in these positions, $T_{12}$ the number of elements from $\{\mu_1 + 1, \dots, \mu_1 + \mu_2 \}$, \dots and $T_{1J}$ the number of elements from $\{n - \mu_{J} + 1, \dots, n \}$. In general, $T_{ij}$ is the number of elements from $\{\mu_1 + \hdots + \mu_{i -1} + 1, \dots, \mu_1 + \hdots + \mu_j \}$ which occur in the positions $\lambda_1 + \lambda_2 + \hdots + \lambda_{i - 1} + 1$ up to $\lambda_1 + \hdots + \lambda_i$.

\begin{example} \label{ex: sigmaT}
When $n = 5$, $\lambda = (3, 2)$, $\mu = (2, 2, 1)$ there are five possible tables:
\begin{align*}
& \begin{pmatrix}
2 & 1 & 0 \cr 
0 & 1 & 1
\end{pmatrix} \,\,\,\,\,\,\,\,\,\,\,  \begin{pmatrix}
2 & 0 & 1 \cr 
0 & 2 & 0
\end{pmatrix}
\,\,\,\,\,\,\,\,\,\,\, \begin{pmatrix}
1 & 2 & 0 \cr 
1 & 0 & 1
\end{pmatrix}
\,\,\,\,\,\,\,\,\,\,\,  \begin{pmatrix}
1 & 1 & 1\cr 
1 & 1 & 0
\end{pmatrix}
\,\,\,\,\,\,\,\,\,\,\,  \begin{pmatrix}
0 & 2 & 1 \cr 
2 & 0 & 0
\end{pmatrix} \\
&\hspace{4mm} \sigma = 12345  \hspace{10mm} \sigma = 12534  \hspace{10mm} \sigma = 13425 \hspace{10mm} \sigma = 13524 \hspace{10mm} \sigma = 34512 \\
&\,\,\,\,\,\,\,\,\,\,\,\, 24  \hspace{23mm} 12  \hspace{23mm} 24 \hspace{22mm} 48 \hspace{22mm} 12
\end{align*}
Listed below each table is a permutation in the corresponding double coset, and the total size of the double coset. The double coset representatives are chosen to be the one of minimal length, i.e.\ the fewest number of adjacent transpositions to change the permutation to the identity. It is always possible to a unique shortest double coset representative \cite{billeyKonvPeterSlofstra}. This is easy to identify: Given $T$, build $\sigma$ sequentially, left to right, by putting down $1, 2, \dots, T_{11}$ then $\mu_1 + 1, \mu_1 + 2, \dots, \mu_1 + T_{12}$ ... each time putting down the longest available numbers in the $\mu_j$ block, in order. In example \ref{ex: sigmaT}, the shortest double coset representatives are shown. 
\end{example}

Thanks to Zhihan Li, another way to describe the mapping from $S_n$ to contingency tables is to consider the $n \times n$ permutation matrix defined by $\sigma$. The partition $\mu$ divides the columns into $J$ blocks and $\lambda$ divides the rows into $I$ blocks. The table $T$ corresponding to $\sigma$ is defined by $T_{ij}$ as the number of $1$s in the $i$ block of rows and $j$ block of columns. For example, $\sigma = 12543$ gives the permutation matrix

\begin{table}[h]
\begin{center}
\begin{tabular}{l || c c | c  c | c }
 &   $\mu_1$ &  &   $\mu_2$ & & $\mu_3$   \\ \hline \hline
$\lambda_1$ &    1 & 0 & 0 & 0 & 0 \\
&    0 & 1 & 0 & 0 & 0  \\ 
 &    0 & 0 & 0 & 0 & 1 \\ \hline
$\lambda_2$ &    0 & 0 & 0 & 1 & 0 \\ 
 &    0 & 0 & 1 & 0 & 0 
\end{tabular}
\end{center}
\end{table}

which defines the table $\begin{pmatrix}
2 & 0 & 1 \cr 
0 & 2 & 0
\end{pmatrix}$.

\subsection{Markov Chain and Main Results}

The \emph{random transpositions} chain on $S_n$ is easy to describe. If a permutation $\sigma \in S_n$ is viewed as an arrangement of a deck of $n$ unique cards, the random transpositions chain can be described: Pick one card with your left hand and one card with your right hand (allowing for the possibility of picking the same card), and swap the two cards. The transition probabilities are generated by the probability measure $Q$ on $S_n$ defined
\begin{equation} \label{eqn: Q}
Q(\sigma) = \begin{cases}
\frac{2}{n^2} & \text{if} \,\,\,\ \sigma = (i, j), i < j \\
\frac{1}{n} & \text{if} \,\,\,\ \sigma = id \\
0 & \text{otherwise},
\end{cases}
\end{equation}
where $(i, j)$ denotes the transposition of elements $i$ and $j$. The transition matrix for the random transposition Markov chain is then
\[
P(x, y) = Q(y \cdot x^{-1}), \,\,\,\,\,\,\,\, x, y \in S_n.
\]
Note that $Q$ is constant on conjugacy classes: If $i < j, k < \ell$ and $x = (ij), y = (k \ell)$ are two transpositions, then
\begin{align*}
    y x y^{-1} = (k \ell) \cdot (ij) \cdot (k \ell) = \begin{cases} 
    (ij) & \text{if} \,\, i \neq k, j \neq \ell  \,\, \text{or} \,\, i = k, j = \ell \\
    (j \ell) & \text{if} \,\, i = k, j \neq \ell \\
    (k i) & \text{if} \,\, j = \ell, i \neq k
    \end{cases}.
\end{align*}
Since transpositions generate $S_n$, this proves that if $x$ is a transposition, then $yxy^{-1}$ is a transposition for any $y \in S_n$, and thus $Q(yxy^{-1}) = Q(x)$.

The random transpositions chain on $S_n$ induces a Markov chain on $\T_{\lambda, \mu}$ as a `lumping'. From a table $T \in \T_{\lambda, \mu}$, choose a permutation $x \in S_n$ that maps to the double coset corresponding to $T$. From $x$, sample $y$ from $P(x, \cdot)$ and move to the contingency table corresponding to $y$. More details of this lumping are discussed in Section \ref{sec: lumped}, and this perspective will be useful for relating the well-studied eigenvalues of the chain on $S_n$ to the chain on $\T_{\lambda, \mu}$. It is also possible to write down the transition probabilities independently of the chain on $S_n$.

\begin{definition}[Random Transpositions Markov chain on $\T_{\lambda, \mu}$] \label{def: chain}
Let $\lambda = (\lambda_1, \dots, \lambda_I), \mu = (\mu_1, \dots, \mu_J)$ be two partitions of $n$. For $T \in \T_{\lambda, \mu}$ and indices $1 \le i_1, i_2 \le I$ and $1 \le j_1, j_2 \le J$, let $F_{(i_1, j_1), (i_2, j_2)}(T) = T'$ denote the table with
\begin{align*}
&T_{i_1, j_1}' = T_{i_1, j_1} - 1, \\
&T_{i_2, j_2}' = T_{i_2, j_2} - 1, \\
&T_{i_1, j_2}' = T_{i_1, j_2} + 1, \\
&T_{i_2, j_1}' = T_{i_2, j_1} + 1,
\end{align*}
and all other entries of $T'$ the same as $T$. 
The random transpositions Markov chain on $\T_{\lambda, \mu}$ is defined by the transition matrix $P$ with:
\[
P(T, T') = \begin{cases}
\frac{2 \cdot T_{i_1, j_1} \cdot T_{i_2, j_2}}{n^2} & \text{if} \,\,\, T' = F_{(i_1, j_1), (i_2, j_2)}(T) \\
1 - \sum_{i_1 < i_2} \sum_{j_1 < j_2} \frac{2 \cdot T_{i_1, j_1} \cdot T_{i_2, j_2}}{n^2} & \text{if} \,\,\, T' = T \\
0 & \text{otherwise}
\end{cases}.
\]
\end{definition}

Note if $T_{i_1, j_1} = 0$ or $T_{i_2, j_2} = 0$, then $F_{(i_1, j_1), (i_2, j_2)}(T)$ has negative values and is not an element of $\T_{\lambda, \mu}$. Definition \ref{def: chain} does not allow selecting this possibility, since only pairs with $T_{i_1, j_1}T_{i_2, j_2} \ge 1$ will be chosen.

There are other ways of thinking about the Markov chain. Suppose the contingency table is created from $n$ data pairs of the type $(i, j)$, where $i$ indicates the row feature and $j$ indicates the column feature. One step of the chain is to pick two pairs uniformly (possibly picking the same pair) and swapping the column values (or equivalently, swapping the row values). This is exactly the \emph{swap Markov chain} on bi-partite graphs, which has been studied for various special cases, e.g.\ \cite{amanatidis2020}.


\begin{example}
$\lambda = (3, 2), \mu = (2, 1, 1)$: \vspace{5mm}
\begin{align*}
&T^x = \begin{pmatrix}
2 & 1 & 0 \cr 
0 & 1 & 1
\end{pmatrix} \hspace{10mm} \longrightarrow  \hspace{10mm} T^y =\begin{pmatrix}
1 & 2 & 0 \cr 
1 & 0 & 1
\end{pmatrix} \\
& \\
&\hspace{5mm} x = 12345 \hspace{16mm} \longrightarrow \hspace{10mm}  y = 1\mathbf{4}3\mathbf{2}5
\end{align*}
The probability of this transition for the tables is $2/5^2$, since there are two possible transpositions in the permutation $x$ that would result in the table $T^y$. Representing the table as a set of $n$ data points, this transition is:
\begin{align*}
T^x: (1, 1), (1, 1), (1, 2), (2, 2), (2, 3) \longrightarrow T^y = (1, 1), (1, \mathbf{2}), (1, 2), (2, \mathbf{1}), (2, 3). 
\end{align*}
\end{example}

\begin{remark}
If $I = J = 2$, then this is related to the \emph{Bernoulli-Laplace} chain: Suppose $\lambda = (n - k, k), \mu = (n - j, j)$. Consider two urns. The first contains $n - k$ total balls and the second has $k$ balls. Of these $n$ total balls, $n - j$ are green and $j$ are red. In the contingency table, the rows represent urns and the columns colors. In the traditional Bernoulli-Laplace urn, one ball is picked uniformly from each urn and swapped. The random transpositions chain Definition \ref{def: chain} can be described by instead picking $2$ balls uniformly from all $n$ balls and swapping (which creates the possibility both are selected from the same urn).

Using the spherical Fourier transform of the Gelfand pair $(S_{2n}, S_n \times S_n)$, Diaconis and Shahshahani proved mixing time of order $(n/4) \log(n)$ steps for the Bernoulli-Laplace chain in a special case (for $2n$ total balls in the system); for thorough analyses of the Bernoulli-Laplace chain, see \cite{diaconisShah}, \cite{nestoridi}, \cite{ceccherini}. An extension of the Bernoulli-Laplace chain with $m$ urns is studied in \cite{scarabotti}; this state space corresponds to $\T_{\lambda, \mu}$ for $\lambda, \mu$ partitions of $mn$ with $\lambda = 1^{mn}, \mu = (n, n, \dots, n)$.
\end{remark}

\begin{remark} \label{remark: holding}
The random transpositions chain is sometimes defined with the probability measure:
\begin{equation*} 
\widetilde{Q}(\sigma) = \begin{cases}
\frac{1}{\binom{n}{2}} & \text{if} \,\,\,\ \sigma = (i, j), i < j \\
0 & \text{else}
\end{cases}.
\end{equation*}
That is, the random walk on $S_n$ is described: pick a card with your left hand and pick a different card with your right hand. This means the permutation will always change at each step, but then the process is periodic. Using $Q$ from \eqref{eqn: Q} (i.e.\ allowing the possibility of picking the same card) the chain is aperiodic, which allows for easier analysis of the chain on $S_n$.

Note that as long as $\lambda, \mu$ are not both equal to $(1, 1, \dots, 1)$ (in which case $S_\lambda \backslash S_n / S_\mu \cong S_n$), the chain using probabilities $1/\binom{n}{2}$ is not periodic: If there is at least one row or column with $2$ entries, e.g.\ $(i, j)$ and $(i, l$), then picking these two entries to swap does not change the table. Thus, using $\widetilde{Q}$ to define the induced process on $\T_{\lambda, \mu}$ is also possible. The analysis and mixing times are the same for both $Q$ and $\widetilde{Q}$, except the chains have slightly different eigenvalues. 
\end{remark}

The interpretation of the chain as a lumping of the random transpositions chain on $S_n$ is advantageous because the chain on contingency tables inherits its eigenvalues from the chain on $S_n$, and these eigenvalues are well-known. Representation theory also gives an expression for the multiplicities of the eigenvalues in terms of \emph{Kostka numbers} corresponding to partitions. Section \ref{sec: lumped} reviews this background in more detail, and thoroughly explains part (b) of the following theorem, which are the results for eigenvalues and multiplicities.

\begin{theorem} \label{thm: ev}
Let $\lambda = (\lambda_1, \dots, \lambda_I), \mu = (\mu_1, \dots, \mu_J)$ be partitions of $n$ and $P$ the random transpositions Markov chain on the space of contingency tables $\mathcal{T}_{\lambda, \mu}$. 
\begin{enumerate}[(a)]
    \item The eigenvalues $\beta_\rho$ are of the form
    \[
    \beta_\rho = \frac{1}{n} + \frac{1}{n^2} \sum_{j = 1}^k \left[ (\rho_j - j)(\rho_j - j + 1) - j(j-1) \right],
    \]
    for $\rho = (\rho_1, \dots, \rho_k)$ a partition of $n$.
    
    \item The multiplicity of $\beta_\rho$ for $P$ is $m_\rho^\lambda \cdot m_\rho^\mu$, where $m_\rho^\lambda$ is the Kostka number (defined below). 
    
    
    \item Eigenfunctions of $P$ are the orthogonal polynomials for the Fisher-Yates distribution. 
\end{enumerate}
\end{theorem}

Theorem \ref{thm: ev} (a) and (b) arise from the general theory for Markov chains on double cosets developed in \cite{diaconisRamSimper} and \cite{diaconis2013random} and summarized in Section \ref{sec: lumped}. Section \ref{sec: evSection} contains a proof of Theorem \ref{thm: ev}(c) as well as a description of the multiplicities. It seems unlikely that there would be a simple formula for the number and multiplicity of the eigenvalues, as this would give a way of enumerating the size of the state space $\T_{\lambda, \mu}$, which is a \#- P complete problem \cite{diaconisGangoli}. The orthogonal polynomials of the Fisher-Yates distribution have not been explicitly computed. Section \ref{sec: poly} computes linear and quadratic eigenfunctions of $P$, for any size table, which can be used as a starting point for finding orthogonal polynomials. 

For $2 \times J$ tables, the stationary distribution is multi-variate hypergeometric, and the computation of the eigenvalues simplifies to give a more complete picture. 

\begin{theorem} \label{thm: 2Jev}
Let $\lambda = (n -k, k)$ for some $k \le \lfloor n/2 \rfloor$, $\mu = (\mu_1, \dots, \mu_J)$, and $P$ be the random swap Markov chain on the space of contingency tables $\mathcal{T}_{\lambda, \mu}$. Then the eigenvalues of $P$ are
\[
\beta_m = 1 - \frac{2 m(n + 1 - m)}{n^2}, \,\,\,\,\,\, 0 \le m \le k,
\]
with eigenbasis the \emph{orthogonal polynomials} of degree $m$ for the multivariate hypergeometric distribution (defined below). The multiplicity of $\beta_m$ is the size of the set 
\[
\left\lbrace(x_1, \dots, x_{J-1}) \in \mathbb{N}^{J-1} : \sum_{j = 1}^{J - 1} x_j = m, x_j < \mu_{J - j + 1} \right\rbrace.
\]
\end{theorem}

The eigenvalues and eigenvectors allow an analysis of the convergence rates of this Markov chain. The measure that we study for convergence rate is the \textit{mixing time}. That is, for a Markov chain on a space $\Omega$ with transition probability $P$ and stationary distribution $\pi$, the mixing time is defined as
\[
t_{mix}(\epsilon) = \sup_{x_0 \in \Omega} \inf \{ t > 0 : \| P^t(x_0, \cdot) - \pi(\cdot) \|_{TV} < \epsilon \}, \quad \epsilon > 0,
\]
where $\| \mu - \pi \|_{TV} = \sup_{A \subset \Omega} |\mu(A) - \pi(A)|$ is the total-variation distance between probability measures.

It is challenging to analyze the mixing time in full generality, but understanding the linear and quadratic eigenfunctions of the chain give the following results: 
\begin{itemize}
\item For $\lambda = (n - k, k), \mu = (n - \ell, \ell)$, the  time to stationarity, averaged over starting positions sampled from $\pi_{\lambda, \mu}$, is bounded above by $(n/4)\log( \min(k, \ell))$.

\item For $\lambda = (n - k, k)$ and any $\mu$, the multivariate Hahn polynomials are used to compute an upper bound for the time to stationarity starting from extreme states in which the second row has a single non-zero entry.

\item For any $\lambda, \mu$, the mixing time is bounded below by $((n/4) - 1)\log(C_{\lambda, \mu}n)$, where $C_{\lambda, \mu}$ is a computable constant. 
\end{itemize}

The mixing time of the random transpositions chain on $S_n$ is $(1/2) n \log(n) + cn$, which is an upper bound for the mixing time of the lumped chain on $\T_{\lambda, \mu}$. The results listed above suggest that there is not much of a speed-up. 


\subsection{Related Work}

Note that for two way contingency tables it is easy to sample directly from the Fisher-Yates distribution (by generating a random permutation and using the argument of Lemma \ref{lem: CTcorrespondence}). One reason for interest in the Markov chain mixing time is that, for three and higher way tables, the Markov chain method is the only available approach, so it is important to see how it works in the two way case.   

The uniform swap Markov chain on contingency tables was first proposed in \cite{diaconisGangoli}, with the motivation of studying the uniform distribution \cite{DEfron}. The mixing time of this chain has been analyzed in special cases, e.g.\  if $n = \sum \lambda_i = \sum \mu_j$ 
and the number of rows and columns is fixed, then the mixing time is of order $n^2$. When the table has two rows, the chain mixes in time polynomial in the number of columns and $n$  \cite{hernek}. Chung, Graham, and Yau \cite{chungGraham} indicate that a modified
version of the chain converges in time polynomial in the number of rows and columns. 

A contingency table can be thought of as a bi-partite graph with prescribed degree sequences, and a similar swap Markov chain defined for any graph with fixed degree sequences to sample from the uniform distribution, e.g.\  \cite{amanatidis2020}. There is a large literature on methods for sampling graphs with given degree sequences (not necessarily bi-partite); see \cite{dutta2021sampling} for a recent review. Recently the chain was studied for contingency tables with entries in $\Z/q\Z$ and mixing time with cut-off proven at $c_q n^2 \log(n)$ \cite{nestoridi2020}.


A Markov chain on $\T_{\lambda, \mu}$ for which the stationary distribution is Fisher-Yates was studied in \cite{diaconisSturmfels}, motivated by studying conditional distributions. The Markov chain is created by taking the Gibbs sampler using the uniform swap chain defined in \cite{diaconisGangoli}. This gives the chain: Pick a pair of rows $i_1, i_2$ and columns $j_1, j_2$ uniformly at random. This determines a $2 \times 2$ submatrix. Replace it with a $2 \times 2$ sub-table with the same margins, chosen from the hypergeometric distribution. As noted in Section 2 of \cite{diaconisSturmfels}, while it is straightforward to sample directly from Fisher-Yates distribution for $2$-way tables, there is not a simple algorithm for sampling from the analagous distribution on $3$-way or higher-dimensional tables. Thus, Markov chains with Fisher-Yates distribution as stationary are valuable for this application; multiway tables are discussed further in Section \ref{sec: multiway}.



\subsection{Outline}
 Section \ref{sec: background} contains an overview of basic definitions and facts for Markov chains, double cosets, and orthogonal polynomials. Section \ref{sec: evSection} establishes the eigenvalues and multiplicities, as well as the eigenfunctions of the Markov chain as polynomials, and gives formulas for the linear and quadratic polynomial eigenfunctions. These are then used in Section \ref{sec: mixingTime} to find explicit upper and lower bounds on the mixing time of the chain, in specific cases. Section \ref{sec: future} discusses some future directions. 

\paragraph{Acknowledgments} The author thanks Persi Diaconis for helpful discussion and suggestions. This work was supported by a National Defense Science \& Engineering Graduate Fellowship and a Lieberman Fellowship at Stanford University.

\section{Background} \label{sec: background}

This section reviews necessary facts and background on Markov chains, orthogonal polynomials, and the random transpositions chain on $S_n$.

 \subsection{Double Coset Markov Chains} \label{sec: lumped}

The results of this section use basic tools of representation theory; in particular, induced representations and Frobenius reciprocity. For background on these topics, see \cite{isaacs}, \cite{jamesLiebeck}, \cite{serre}. In the specific case in hand -- the symmetric group -- additional tools are available; Young's rule and the hook-length formula. These are clearly developed in \cite{jamesBook}. See also, \cite{jamesKerber}.

Let $H, K$ be subgroups of the finite group $G$, $Q$ a probability on $G$. If $\text{supp}(Q)$ is not contained in a coset of a subgroup, then the random walk on $G$ induced by $Q$ is ergodic with a uniform stationary distribution. This random walk on $G$ is defined by picking a random element from $Q$ and multiplying the current state. That is, transitions are
\[
P(x, y) = Q(yx^{-1}).
\]
See \cite{diaconisBook} for an introduction to random walks on groups.

The double cosets of $H \backslash G / K$ partition the space $G$ and any Markov chain on $G$ defines a random process on the set of double cosets by keeping track of which double coset the process on $G$ is in, giving a `lumped' process. Proposition \ref{prop: inducedChain} gives a condition for when the random walk induced by $Q$ on $G$ is also a Markov chain on the double cosets. For the statement we fix double coset representatives and write $x$ for the double coset $HxK$. The result is proven in \cite{diaconisRamSimper}.


\begin{proposition} \label{prop: inducedChain}
Let $Q$ be a probability on $G$ with is $H$-conjugacy invariant ($Q(s) = Q(h^{-1}s h) $ for $h \in H, s \in G$). Then, the image of the random walk driven by $Q$ on $G$ maps to a Markov chain on $H \backslash G /K$ with transition kernel
\[
P(x, y) = Q(HyKx^{-1}).
\]
The stationary distribution is $\pi(x) =  |HxK|/|G|$. If $Q(s) = Q(s^{-1})$, then $(P, \pi)$ is reversible.
\end{proposition}

\begin{remark}
It is special for a function of a Markov chain to also be Markov. In this setting, many of the famous Markov chains on $S_n$ are not Markov when they are lumped to the the double cosets $S_\lambda \backslash S_n / S_\mu$. For example, the adjacent transpositions random walk on $S_n$ is induced by $
Q((i, i+1)) = 1/(n-1)$, which does not satisfy the conditions of Proposition \ref{prop: inducedChain}. See \cite{pang} for a survey on lumped Markov chains.
\end{remark}

Since transpositions form a conjugacy class in $S_n$, the probability $Q$ from \eqref{eqn: Q} satisfies Proposition \ref{prop: inducedChain}, so lumped to contingency tables gives a Markov process. The following lemma also directly proves that the lumping of the random transpositions chain to contingency tables is Markov, and equivalent to Definition \ref{def: chain}. 

\begin{lemma}
The random transpositions chain on $S_n$ induced by $Q$ from \eqref{eqn: averageBound} when lumped to double cosets $S_\lambda \backslash S_n / S_\mu$ is equivalent to the Markov chain defined in Definition \ref{def: chain}.
\end{lemma}

\begin{proof}

Suppose $\lambda = (\lambda_1, \dots, \lambda_I), \mu = (\mu_1, \dots, \mu_J)$. Define $L_1 = \{1, \dots, \lambda_1 \}, L_2 = \{\lambda_1 + 1, \dots, \lambda_1 + \lambda_2 \}, \dots L_I = \{n - \lambda_I, \dots, n-1, n \}$ and similarly use $\mu$ to define $M_1, \dots, M_J$. Then the contingency table corresponding to the double coset of $x \in S_n$ is defined by
\[
T^x = \{T_{ij}^x \}, \quad T_{ij}^x = \# \{ \ell \in L_i : x(\ell) \in M_j \}.
\]

Let $(ab)$ be a transposition in $S_n$. If $y = (ab)x$ then $y(a) = x(b), y(b) = x(a)$ and $y(c) = x(c), c \neq a, b$. Suppose $a \in L_{i_1}, b \in L_{i_2}, x(a) \in M_{j_1}, x(b) \in M_{j_2}$. If $i_1 = i_2$ or $j_1 = j_2$ then $y \in S_\lambda x S_\mu$. Otherwise, $T^y = F_{(i_1, j_1), (i_2, j_2)}(T^x)$, as defined in Definition \ref{def: chain}.

Note that if $x, x' \in S_n$ are contained in the same double coset, i.e.\ $T^x = T^{x'}$, then  for any $y \in S_n$
\[
\sum_{y' \in S_\lambda y S_\mu} Q(y' x^{-1}) = \sum_{y' \in S_\lambda y S_\mu} Q(y' (x')^{-1}).
\] 
In words, for the random transpositions chain on $S_n$, the probability of transitioning from the double coset $S_\lambda x S_\mu$ to another double coset $S_\lambda y S_\mu$ doesn't depend on the choice of double coset representative. This is Dynkin's condition for lumped Markov chains (\cite{levinPeres}, \cite{pang}).

\end{proof}

\paragraph{Eigenvalues} Let $1 = \beta_0 > \beta_1 \ge \hdots \ge \beta_{|\Omega| - 1} \ge -1$ be the eigenvalues of $P$ with corresponding eigenfunctions $f_0, f_1, \dots, f_{|\Omega| -1}$. That is, $f_i: \Omega \to \R$ and
\[
\E[f_i(X_1) \mid X_0 = x] = \beta_i \cdot f_i(x), \,\,\,\,\, x \in \Omega.
\]
For reversible Markov chains, the eigenfunctions $\{f_i\}_{i \ge 0}$ can be chosen to be \emph{orthonormal} with respect to the stationary distribution:
\[
\sum_{x \in \Omega} f_i(x) \cdot f_j(x) \cdot \pi(x) = \textbf{1}(i = j). 
\]
The eigenvalues and eigenfunctions give an exact formula for the \emph{chi-squared} distance between the chain and the stationary distribution, defined
\[
\chi^2_x(t) := \sum_{y \in \Omega} \frac{\left| P^t(x, y) - \pi(y) \right|^2}{\pi(y)}.
\] 
The chi-squared distance then gives an upper bound for the total variation distance. This information is summarized in the following lemma, see \cite{levinPeres} Chapter 12.

\begin{lemma} \label{lem: evChiSquare}
For any $x \in \Omega$ and $t > 0$, 
\[
\|P^t(x, \cdot) - \pi(\cdot) \|_{TV}^2 \le \frac{1}{4}\chi^2_x(t) = \frac{1}{4}\cdot \sum_{i=1}^{|\Omega| - 1} \beta_i^{2t} \cdot f_i^2(x).
\]
\end{lemma}

Note that the worst-case mixing time is defined by looking at the total variation distance from the worst-case starting point. Thus to bound the mixing time it is needed to be able to analyze $f_i^2(x)$ for any $x \in \Omega$. Even when the eigenfunctions are explicitly known, this can be challenging. For chains $P(x,y)$ on $\Omega$ where a group $G$ acts transitively (for all $x, y \in \Omega, g \in G$, $P(x,y)= P(gx,gy)$), the distance from all starting states is the same and the right hand side of the bound in Lemma \ref{lem: evChiSquare}  is $\sum_{i} \beta_i^{2t}$. Another bound, for any starting state $x$, is 
\begin{equation} 
    4 \|P_x^\ell - \pi \|_{TV}^2 \le \frac{1}{\pi(x)} \beta_{*}^{2 \ell},
\end{equation}
with $\beta_* = \max_{\lambda \neq 1} |\beta_\lambda|$. If $\pi(x)$ is not uniform this bound can vary widely, as in the following example. 

\begin{example} \label{ex: transvections}
The paper \cite{diaconisRamSimper} analyzes the random transvections chain on $GL_n(\mathbb{F}_q)$ lumped to double cosets $\mathcal{B} \backslash GL_n(\mathbb{F}_q) / \mathcal{B}$, with $\mathcal{B}$ the Borel subgroup of upper triangular matrices. Here double cosets are indexed by permutations in $S_n$, and the uniform distribution on $GL_n(\mathbb{F}_q)$ induces the Mallows distribution on $S_n$ with parameter $q$:
 \[
p_q(\omega) = \frac{q^{I(\omega)}}{[n]_q!}, \quad [n]_q! := (1 + q)(1 + q + q^2)\hdots (1 + q + \hdots + q^{n-1}),
 \]
 where $I(\omega)$ is the number of inversions in $\omega$. In the setting $q > 1$, the reversal permutation has largest probability and the identity has smallest. Transvections form a conjugacy class in $GL_n(\mathbb{F}_q)$, and so character theory gives the eigenvalues and the mixing time from special starting states can be analyzed. Starting from the identity the mixing time is order $n$ (the same order as the random transvections walk on $GL_n(\mathbb{F}_q)$), but starting from the reversal permutation the mixing time is order $\log(n)$.
 \end{example}

 The analysis in Example \ref{ex: transvections} is amenable because, in the setting of Proposition \ref{prop: inducedChain}, the measure $Q$ was a class function, i.e.\ $Q(s) = Q(t^{-1}st)$ for all $s, t, \in G$. In this case, the eigenvalues of the walk on $G$ are expressed using the characters of the irreducible complex representations of $G$. Let $\widehat{G}$ be an index set for these representations, $\lambda \in \widehat{G}$, and $\chi_\lambda$ the character of $\lambda$. The restriction of $\chi_\lambda$ to $H$ is written $\chi_\lambda|_H$ and $\left\langle \chi_\lambda |_H, 1 \right\rangle$ is the number of times the trivial representation of $H$ appears in $\chi_\lambda|_H$. By reciprocity, this is $\left\langle \chi_\lambda, \mathrm{Ind}_H^G(1) \right\rangle$. 
 
 The following theorem is proven in \cite{diaconisRamSimper}.

\begin{proposition} \label{prob: DCev}
Let $Q$ be a class function on $G$ and let $H, K$ be subgroups of $G$. The induced chain $P(x, y)$ of Proposition \ref{prop: inducedChain} on $H \backslash G / K$ has eigenvalues
\begin{equation}
    \beta_\lambda = \frac{1}{\chi_\lambda(1)} \sum_{s \in G} Q(s), \quad \lambda \in \widehat{G}
\end{equation}
with multiplicity
\begin{equation}
m_\lambda = \left\langle \chi_\lambda|_H, 1 \right\rangle \cdot \left\langle \chi_\lambda|_K, 1 \right\rangle.
\end{equation}
Further, the average square chi-squared distance to stationarity satisfies
\begin{equation} \label{eqn: averageBound}
    \sum_{x\in \Omega} \pi(x) \chi^2_x(\ell) = \frac{1}{4} \sum_{\lambda \in \widehat{G}, \lambda \neq 1} m_\lambda \beta_\lambda^{2 \ell}.
\end{equation}
\end{proposition}

\subsection{Random Transpositions on $S_n$}

The driving Markov chain on $S_n$ for contingency tables is the random transpositions Markov chain studied in \cite{diaconisShahshahani1981}. This uses the tools of representation theory and the character theory of the symmetric group. An expository account, aimed at probabilists and statisticians is in Chapter 3D in \cite{diaconisBook}. One of the main results used below is an explicit determination of the eigenvalues of this chain.

Recall the probability measure which defines the random walk on $S_n$:
\begin{equation*} 
Q(\sigma) = \begin{cases}
\frac{2}{n^2} & \text{if} \,\,\,\, \sigma = (i, j), i < j \\
\frac{1}{n} & \text{if} \,\,\,\, \sigma = id, \\
0 & \text{otherwise},
\end{cases}.
\end{equation*}
This measure $Q$ is not concentrated on a single conjugacy class (the identity is not in the conjugacy class of transpositions). However, we can write
\[
Q = \frac{1}{n}I + \frac{n - 1}{n}\widetilde{Q},
\]
where $I(\sigma) = \mathbf{1}(\sigma = id)$ and $\widetilde{Q} = \frac{1}{\binom{n}{2}} \mathbf{1}(\sigma \in \mathcal{C})$, where $\mathcal{C} \subset S_n$ denotes the conjugacy class of transpositions. (This $\widetilde{Q}$ is the same as the one discussed in Remark \ref{remark: holding}.) Then $\widetilde{Q}$ is concentration on a single conjugacy class $\mathcal{C}$, so the eigenvalues of the random walk are equal to the character ratio,
\[
\widetilde{\beta}_\rho = \frac{1}{\chi_\rho(1)} \sum_{s \in G} Q(s) = \frac{\chi_\rho(\mathcal{C})}{\chi_\rho(1)},
\]
for each partition $\rho$ of $n$.. The formula for this character ratio was determined by Frobenius; see \cite{macdonald} for a modern exposition, and \cite{ingram} for a proof of this special case:
\[
\widetilde{\beta}_\rho = \frac{1}{n(n-1)} \sum_{j = 1}^k \left[ \rho_j^2 - (2j - 1) \rho_j \right].
\]
The eigenvalues for the walk driven by $Q$ are then related:
\[
\beta_\rho =  \frac{1}{n} + \frac{n - 1}{n}\widetilde{\beta}_\rho.
\]
This information is summarized in the following lemma; see \cite{diaconisShahshahani1981} for more details. 

\begin{lemma}[\cite{diaconisShahshahani1981}, Corollary 1 \& Lemma 7] \label{lem: RTSn}
If $P$ is the random transpositions chain on $S_n$ driven by $Q$, then $P$ has an eigenvalue $\beta_\rho$ for each partition $\rho = (\rho_1, \rho_2, \dots, \rho_k)$ of $n$. The eigenvalue corresponding to $\rho$ is
\[
\beta_\rho = \frac{1}{n} + \frac{1}{n^2} \sum_{j = 1}^k \left[ \rho_j^2 - (2j - 1) \rho_j \right].
\]

\end{lemma}

In \cite{diaconisShahshahani1981}, the chain is proven to mix in total variation distance, with cut-off, after $(n/2) \log(n)$. This gives an initial upper bound on the mixing time of the random transpositions chain on contingency tables. 



\subsection{Orthogonal Polynomials}

This section reviews orthogonal polynomials and records the formula for multivariate Hahn polynomials. See \cite{dunklXuBook} for a thorough exposition on multivariate orthogonal polynomials; especially Section 3 for sufficient conditions for an orthonormal basis to exists for a probability distribution. 

Let $\pi$ be a probability distribution on a finite space $\Omega \subset \N^d$. Let $\ell^2(\pi)$ be the space of functions $f: \Omega \to \R$ with the inner product
\[
\left\langle f, g \right\rangle_\pi = \E_\pi[f(X) g(X)] = \sum_{x \in \Omega} f(x)g(x) \pi(x).
\]

 A set of functions $\{q_m\}_{0 \le m < |\Omega|}$ are \emph{orthogonal} in $\ell^2(\pi)$ if
\[
\left\langle q_m, q_\ell \right\rangle_\pi = d_m^2 \textbf{1}(m = \ell).
\]


For the following lemma, $\mathbf{m} = (m_1, \dots, m_{N})$ denotes an index vector and $|\mathbf{m}| = \sum m_i$ is the total degree of the polynomial defined by the vector. 
\begin{lemma}[\cite{khare}, Lemma 3.2] \label{lem: khare}
 Suppose $\pi$ is a distribution on the space $\Omega \subset \mathbb{N}^N$, for some $N$, and $\{q_{\mathbf{m}}\}$ is an orthogonal basis of $\ell^2(\pi)$, where $q_{\mathbf{m}}$ is a polynomial of exact degree $|\mathbf{m}|$. Let $(X_t)_{t \ge 0}$ be a reversible Markov chain with transition matrix $P$ and stationary distribution $\pi$. Suppose that, 
\[
\E \left[ X_1^{\mathbf{m}} \mid X_0 = \mx \right] = \beta_{|\mathbf{m}|} x^{\mathbf{m}} + \left( \text{terms in x of degree} \,\,\, < |\mathbf{m}| \right).
\]
Then $P$ has eigenvalue $\beta_{|\mathbf{m}|}$ with eigenbasis $\{q_{\mathbf{m}} \}_{|\mathbf{m}| = m}$.

Write $\beta_m = \beta_{|\mathbf{m}|}$ for $|\mathbf{m}| = m$. The chi-square distance between $P^t(\mx, \cdot)$ and $\pi$ is
\[
\chi^2_\mx(t) = \sum_{m \ge 1} \beta_m^{2 t} \cdot h_m(\mx, \mx),
\]
where
\[
h_m(\mx, \my) := \sum_{|\mm| = m} q_{\mm}(\mx) q_{\mm}(\my) \left\langle q_{\mm}, q_{\mm} \right\rangle_{\pi}^{-1}
\]
is called the kernel polynomial of degree $m$.
\end{lemma}

\paragraph{Multivariate Hahn Polynomials} 
For contingency tables with only two rows, the Fisher-Yates distribution is simply the multivariate hypergeometric distribution. That is, if $\lambda = (n - k, k), k < \lfloor n/2 \rfloor$ and  $\mu = (\mu_1, \dots, \mu_J)$, then a contingency table in $\T_{\lambda, \mu}$ can be represented by a $(J - 1)$-dimensional vector in the space
\[
\mathbf{X}_{k, \mu}^J = \{\mx = (x_1, \dots, x_J) \in \N_0^J : |\mx| = k, x_j \le \mu_j \},
\]
where $|\mx| = \sum_i x_i$. The \emph{multivariate hypergeometric distribution} is
\[
H_{k, \mu}(\mx) = \frac{\prod_{j = 1}^J \binom{\mu_j}{x_j}}{\binom{n}{k}},
\]
where $n = \sum_j \mu_j$. For example, this distribution arises from sampling without replacement: An urn contains $n$ balls of $|\mu|$ different colors, $\mu_j$ of color $j$. If $\mathbf{X}$ is a vector which records the number of each color drawn in a sample (without replacement) of size $k$, then $\mathbf{X}$ has the multivariate hypergeometric distribution. The orthogonal polynomials for this distribution are called the \emph{multivariate Hahn polynomials}. An overview of univariate Hahn polynomials is in \cite{ismail}. Following the notation from \cite{khare} (Section 2.2.1) and \cite{iliev}, define


\begin{align*}
&a_{(k)} = a(a + 1) \hdots (a + k -1), \\
&a_{[k]} = a(a - 1) \hdots (a - k + 1), \\
&a_{(0)} = a_{[0]} = 1.
\end{align*}
For a vectors $\mx = (x_1, \dots, x_d)$,
\begin{align*}
|\mx| = \sum_{i = 1}^d x_i, \quad |\mx_i| = \sum_{j = 1}^i x_j, \quad |\mx^i| = \sum_{j = i}^d x_j.
\end{align*}

The multivariate \emph{Hahn polynomials}, defined in Section 2.2.1 of \cite{khare}, are
\begin{equation}
Q_{\mm}(\mx; k, \mu) = \frac{(-1)^{|\mm|}}{(k)_{[|\mm|]}} \prod_{j = 1}^{J - 1} \left( -k + |\mx_{j-1}| + |\mm^{j+1}| \right)_{(m_j)} \cdot Q_{m_j}\left( x_j; k - |\mx_{j-1}| - |\mm^{j+1}|, -\mu_j, |\mu^{j+1}| + 2|\mm^{j+1}| \right),
\end{equation}
where
\begin{align*}
Q_m(x; k, \alpha, \beta) &= \prescript{}{3}F_2 \begin{pmatrix}
-n, n + \alpha + \beta - 1, - x & \vline & 1 \\
\alpha, -N & \vline 
\end{pmatrix}
\end{align*}
is the classical univariate Hahn polynomial.

For calculating the expression for chi-squared distance, if the orthogonal polynomials are the eigenfunctions, then it is the kernel polynomials which need to be solved for. These were constructed by Griffiths \cite{griffiths}, \cite{griffithsKernel}. It is an open problem to find a useful equation for the kernel polynomials evaluated at more general states.

\begin{proposition}[Proposition 2.6 from \cite{khare}] \label{prop: kernel}
Suppose $\mu_j \ge k$ for all $j$. Let $\mathbf{e}_j$ be the vector with $1$ in the $j$th index and $0$ elsewhere. Then, 
\begin{equation}
h_m(k \mathbf{e}_j, k \mathbf{e}_j) = \binom{k}{m} \frac{(n - 2m + 1)n_{[m - 1]}(n - \mu_j)_{[m]}}{(n - k)_{[m]} (\mu_j)_{[m]}}.
\end{equation}
\end{proposition}

\subsection{Fisher-Yates Distribution} \label{sec: FYdistribution}


The measure induced on contingency tables by the uniform distribution on $S_n$ is
\begin{equation} \label{eqn: fisherNew}
    \pi_{\lambda, \mu}(T) = \frac{1}{n!} \prod_{i, j} \frac{\lambda_i! \mu_j!}{T_{ij}!}.
\end{equation}

This is the \emph{Fisher-Yates} distribution on contingency tables, a mainstay of applied statistical work in chi-squared tests of independence. In applications it is natural to test the assumption that row and column features are independent and that the observed table is a sample with cell probabilities $p_{ij} = \theta_i \eta_j$. For this model, the row and column sums are sufficient statistics; conditioning on the row and column sums of the table gives the Fisher-Yates distribution. 

A useful way of describing a sample from $\pi_{\lambda, \mu}$ is by \emph{sampling without replacement}: Suppose that an urn contains $n$ total balls of $J$ different colors, $\mu_j$ of color $j$. To empty the urn, make $I$ sets of draws of unequal sizes.  First draw $\lambda_1$ balls, next $\lambda_2$, and so on until there are $\lambda_I = n - \sum_{i = 1}^{I - 1} \lambda_i$ balls left. Create a contingency table by setting $T_{ij}$ to be the number of color $j$ in the $i$th draw.

This description is useful for calculating moments of entries in a table, by utilizing the fact that the draws are exchangeable: Let $X_s^{(i, j)}$, $1 \le s \le n$ be $1$ if in the $i$th round of drawings, the $s$th draw was color $j$, and $0$ otherwise. That is,
\[
T_{ij} = \sum_{s = 1}^{\lambda_i} X_s^{(i, j)}.
\]
The expectation of one entry in the table is
\begin{align*}
\E_{\pi_{\lambda, \mu}}[T_{ij}] = \sum_{s = 1}^{\lambda_i}\P( X_s^{(i, j)} = 1) = \frac{\lambda_i \mu_j}{n},
\end{align*}
using that $\P( X_s^{(i, j)} = 1) = \mu_j/n$ for any $s$ since the variables are exchangeable in $s$. These calculations for the moments of Fisher-Yates tables are needed in the following section to normalize eigenfunctions.

The usual test of the independence model uses the chi-squared statistic:
\[
\chi^2(T) = \sum_{i, j} \frac{(T_{ij} - \lambda_i \mu_j/n)^2}{\lambda_i \mu_j/n}.
\]
Under general conditions, see e.g.\ \cite{kangKlotz}, the chi-squared statistic has limiting limiting chi-squared distribution with $(I-1)\cdot(J-1)$ degrees of freedom. This result says that, if the null hypothesis is true, most tables will be close to a table $T^*$ with entries $T^*_{ij} = \lambda_i \mu_j/n$. Another feature to investigate for the independence model is the number of zero entries in a contingency table. Since most tables will be close to the table $T^*$, which has no zeros, zeros are a pointer to the breakdown of the independence model.  Section 5.2 of \cite{diaconisSimper} proves that the number of zeros is asymptotically Poisson, under naturally hypotheses on row and column sums. In \cite{paguyo}, limit theorems for fixed points, descents, and inversions of permutations chosen uniformly from fixed double cosets are studied.  




\paragraph{Majorization Order} Let $T$ and $T'$ be tables with the same row and column sums. Say that $T \prec T'$ (`$T'$ majorizes $T$') if the largest element in $T'$ is greater than the largest element in $T$, the sum of the two largest elements in $T'$ is greater than the sum of the two largest elements in $T$, and so on. Of course the sum of all elements in $T'$ equals the sum of all elements of $T$.

\begin{example} \label{ex: major}
For tables with $n = 8, \lambda_1 = \lambda_2 = \mu_1 = \mu_2 = 4$, there is the following ordering
\[
\begin{pmatrix}
2 & 2 \cr 
2 & 2 
\end{pmatrix} \prec \begin{pmatrix}
3 & 1 \cr 
1 & 3 
\end{pmatrix} \prec
\begin{pmatrix}
4 & 0 \cr 
0 & 4 
\end{pmatrix}.
\]
\end{example}

Majorization is a standard partial order on vectors \cite{marshallOlkinBooks} and Harry Joe \cite{joe1985ordering} has shown it is useful for contingency tables.

\begin{proposition} \label{prop: CTorder}
Let $T$ and $T'$ be tables with the same row and column sums given by $\lambda, \mu$ and $\pi_{\lambda, \mu}$ the Fisher-Yates distribution. If $T \prec T'$, then \[
\pi_{\lambda, \mu} (T) >  \pi_{\lambda, \mu}(T').
\]
\end{proposition}

\begin{proof}
From the definition, we have $\log(pi_{\lambda, \mu}(T)) = C - \sum_{i, j} \log(T_{ij}!)$ for a constant $C$. This form makes it clear the right hand side is a symmetric function of the $IJ$ numbers $\{T_{ij} \}$. The log convexity of the Gamma function shows that it is concave. A symmetric concave function is Schur concave: That is, order-reversing for the majorization order \cite{marshallOlkinBooks}.
\end{proof}

\begin{remark} \label{remark: largest}
Joe \cite{joe1985ordering} shows that, among the real-valued tables with given row and column sums, the independence table $T^*$ is the unique smallest table in majorization order. He further shows that if an integer valued table is, entry-wise, within $1$ of the real independence table, then $T$ is the unique smallest table with integer entries. In this case, the corresponding double coset has $P(T)$ largest. 
\end{remark}

\begin{example}
Fix a positive integer $a$ and consider an $I \times J$ table $T$ with all entries equal to $a$. This has constant row sums $J\cdot a$ and column sums $I \cdot a$. It is the unique smallest table with these row and column sums, and so corresponds to the largest double coset. For $a = 2, I = 2, J = 3$, this table is
\[
T = \begin{pmatrix} 2 & 2 & 2 \cr 2 & 2 & 2 \end{pmatrix}.
\]
\end{example}

Contingency tables with fixed row and column sums form a graph with edges between tables that can be obtained by one move of the following: pick two rows $i, i'$ and two columns $j, j'$. Add $+1$ to the $(i, j)$ entry, $-1$ to the $(i', j)$ entry,  $+1$ to the $(i', j')$ entry, and $-1$ to the $(i, j')$ entry (one move of the Markov chain Definition \ref{def: chain}). This graph is connected and moves up or down in the majorization order as the $2 \times 2$ table with rows $i, i'$ and columns $j, j'$ moves up or down. See Example \ref{ex: major} above.

\section{Eigenvalues and Eigenfunctions} \label{sec: evSection}

From Proposition \ref{prob: DCev}, the multiplicity of the eigenvalue $\beta_\rho$ in the contingency table chain is
\begin{align*}
&m_\rho = m_\rho^\lambda \cdot m_\rho^\mu \\
&m_\rho^\lambda := \left\langle \chi_\rho, \mathrm{Ind}_{S_\lambda}^{S_n}(1) \right\rangle 
\end{align*}
These multiplicities can be determined by Young's Rule \cite{diaconisWald}, \cite{jamesBook}: Given $\lambda$ a partition of $n$, take $\lambda_1$ of the symbol `$1$', $\lambda_2$ of the symbol `2', and so on. The value $\left\langle \chi_\rho, \mathrm{Ind}_{S_\lambda}^{S_n}(1) \right\rangle$ is equal to the number of ways of arranging these symbols into a tableau of shape $\rho$ with strictly increasing columns and weakly increasing rows. In other words, $m_\rho^\lambda$ is the number of semistandard Young tableau of shape $\rho$ and weight $\lambda$. These are also called \emph{Kostka numbers}, and have a long enumerative history in connection with symmetric functions (see Chapter 7 of \cite{stanley}).

\begin{example}
Consider $n = 5$ and $\lambda = (3, 1, 1), \mu = (2, 2, 1)$. The possible eigenvalues and their multiplicities are in Table \ref{tab: evs}.

\begin{table}[h] 
\begin{center}
\begin{tabular}{|l|l|l|l|l|}
\hline
\textbf{$\rho$} & \textbf{$\beta_\rho$} & $m_\rho^\lambda$ & $m_\rho^\mu$ & \textbf{$m_\rho$} \\ \hline \hline
$(5)$           & $1.0$                 & 1                                                                & 1                                                            & 1                 \\ \hline
$(4, 1)$        & $0.6$                 & 2                                                                & 2                                                            & 4                 \\ \hline
$(3, 2)$        & $0.36$                & 1                                                                & 2                                                            & 2                 \\ \hline
$(3, 1, 1)$     & $0.20$                & 1                                                                & 1                                                            & 1                 \\ \hline
\end{tabular}
\end{center}
\caption{Eigenvalues and multiplicities for $\lambda = (3, 1, 1), \mu = (2, 2, 1)$.} \label{tab: evs}
\end{table}

For example, for $\mu = (2, 2, 1)$ the symbols are $11223$ and there are $2$ ways of arranging them into a tableau of shape $(3, 2)$ with increasing rows and strictly increasing columns:
\[
\begin{matrix}
1 & 1 & 2 \cr
2 & 3 & 
\end{matrix}, \quad
\begin{matrix}
1 & 1 & 3 \cr
2 & 2 & 
\end{matrix}.
\]

\end{example}

Thus for $\rho = (3, 2)$, the contribution to the multiplicity is $m_\rho^\mu = 2$, as shown in Table \ref{tab: evs}. While there is not an explicit formula for the result of Young's rule (indeed, this would provide a formula for the number of contingency tables with fixed row and column sums), we can say things for special cases. See the exercises in Chapter 6.4 \cite{macdonald} for more properties.

\begin{corollary} \label{cor: multCases}
Let $\rho, \lambda, \mu$ be partitions of $n$ and $|\rho|$ denote the number of parts of the partition, with $|\lambda| = I, |\mu| = J$. Let $\beta_\rho$ denote the eigenvalue in the random transpositions chain corresponding to $\rho$ and $m_\rho$ the multiplicity of $\beta_\rho$ in the chain lumped to $\T_{\lambda, \mu}$.
\begin{enumerate}[(a)]
    \item The multiplicity of the second-largest eigenvalue $\beta_{(n-1, 1)} = 1 - 2/n$ is $m_{(n-1, 1)} = (I -1)\cdot(J-1)$.
    
    \item If $m_\rho > 0$, then $|\rho| \le \min(I, J)$ and $\rho_1 \ge \max(\lambda_1, \mu_1)$. 
    
    \item If $\rho = (n -k, k)$ for $1 \le k \le \lfloor n/2 \rfloor$ and $\lambda = (n - j, j)$ then 
    \[
    m_\rho^\lambda = \begin{cases}
    0 & \text{if} \,\,\, k > j \\
    1 & \text{else}
    \end{cases}.
    \]
       
      \item  If $\rho = (n - k, k)$, $\mu = (\mu_1, \dots, \mu_J)$ and $\mu_1 \ge k$, then
    \[
    m_\rho^\mu = \left\lbrace(x_1, \dots, x_{J-1}) \in \mathbb{N}^{J-1} : \sum_{j = 1}^{J - 1} x_j = k, x_j \le \mu_{j + 1} \right\rbrace = |\T_{(k, n - \mu_1 - k), (\mu_2, \mu_3, \dots, \mu_j)}|.
    \]
    
    \end{enumerate}
\end{corollary}

\begin{remark} 
Corollary \ref{cor: multCases}(b) shows that a table with only two rows or columns will only have eigenvalues $\beta_\rho$ with $\rho = (n - m, m)$. The eigenvalue defined by $\rho = (1, 1, \dots, 1)$ is $\beta_\rho = -1$. This has $0$ multiplicity unless $\lambda = \mu = (1, 1, \dots, 1)$ for which $\T_{\lambda, \mu} \simeq S_n$. 
\end{remark}

\begin{proof}
\textbf{(a):} An arrangement of symbols determined by $\lambda$ into a tableau of shape $(n-1, 1)$ is determined by the choice of the symbol to be in the second row. To fit the constraint the columns in the tableau are strictly increasing, any symbol except `1' could be placed in the second row. Thus, there are $|\lambda| - 1$ possibilities. For example, if $\lambda = (3, 2, 1)$, there are 2 possibilities:
\begin{align*}
\begin{matrix}
1 & 1 & 1 & 2 & 2 \cr
3 &   & & &
\end{matrix}, \quad 
\begin{matrix}
1 & 1 & 1 & 2 & 3 \cr
2 &   & & &
\end{matrix}.
\end{align*}

\textbf{(b):} Consider the first column of an arrangement of the symbols determined by $\lambda$ into a tableau of shape $\rho$. For the column to be strictly increasing, there must be at least $|\rho|$ symbols from $\lambda$, to give the column $1, 2, \dots, |\rho|$. Thus, if $|\lambda| < |\rho|$ then $m_\rho^\lambda = 0$. All of the `1' symbols must be contained in the first row of the tableau, which gives the constraint $\lambda_1 \le \rho_1$. 

\textbf{(c):} If $|\lambda| = 2$ and $|\rho| = 2$ then there is at most one way of arranging the symbols of $\lambda$ into a tableau of shape $\mu$. For example, $\lambda = (3, 2), \rho = (4, 1)$:
\[
\begin{matrix}
1 & 1 & 1 & 2 \\
2 &   & & 
\end{matrix}.
\]
To satisfy the constraint that columns are strictly increasing it is necessary that the second row only contains symbols `2'. If $\lambda_1 > \rho_1$, this is not possible. Note that the assumption $\lambda_2 \le \lambda_1$ ensures that there would never need to be a column with only the symbol `2'.

\textbf{(d):} Now suppose $\mu = (\mu_1, \dots, \mu_J)$ with $J > 2$. The assumption $\mu_1 \ge k$ ensures that any assignment of the second row will obey the strictly increasing column constraint. That is, any selection of $k$ symbols from the $n - \mu_1$ symbols which are greater than $1$ would be a valid assignment for the second row of the tableau. There are $J - 1$ possible symbols, $\mu_j$ of each type. If $x_j$ denotes the number of symbol $j+1$ in row $2$, then the second row is determined by $(x_1, \dots, x_{J-1})$ with $x_j \le \mu_{j+1}$ and $\sum_{j = 1}^{J-1} x_j = k$. The number of possibilities is exactly the number of $2 \times (J-1)$ contingency tables with row sums $k, n - \mu_1 - k$ and column sums $\mu_2, \dots, \mu_J$.

\end{proof}

The multiplicities also behave well with respect to the \emph{majorization order} on partitions: If $\lambda, \lambda'$ are partitions of $n$, then $\lambda \prec \lambda'$ if $\lambda_1 + \lambda_2 + \hdots + \lambda_k \le \lambda_1' + \lambda_2' + \hdots + \lambda_k'$ for all $1 \le k \le |\lambda|$. In other words, $\lambda'$ can be obtained from $\lambda$ by successively `moving up boxes' (\cite{macdonald} Chapter 1). For example,
\begin{center}
\begin{tikzpicture}
\draw (-.5, -.5) node {$\lambda = $ };
\foreach \x in {0,...,3}
	\filldraw (\x*.25, 0) circle (.5mm);
\foreach \x in {0,...,3}
	\filldraw (\x*.25, -.25) circle (.5mm);
\foreach \x in {0,...,2}
	\filldraw (\x*.25, -.5) circle (.5mm);
\foreach \x in {0,...,0}
	\filldraw (\x*.25, -.75) circle (.5mm);
	\filldraw [red] (1*.25, -.75) circle (.5mm);

\draw (2, -.5) node {$\lambda^\prime = $ };
\foreach \x in {0,...,3}
	\filldraw (\x*.25 + 2.75, 0) circle (.5mm);
	\filldraw [red] (4*.25 + 2.75, 0) circle (.5mm);
\foreach \x in {0,...,3}
	\filldraw (\x*.25 + 2.75, -.25) circle (.5mm);
\foreach \x in {0,...,2}
	\filldraw (\x*.25 + 2.75, -.5) circle (.5mm);
\foreach \x in {0,...,0}
	\filldraw (\x*.25 + 2.75, -.75) circle (.5mm);

\end{tikzpicture}
\end{center}

\begin{remark}
The eigenvalues $\beta_\rho$ are monotonic with respect to this ordering: If $\rho \prec \rho'$, then $\beta_\rho \le \beta_{\rho'}$; see Chapter 3 of \cite{diaconisBook} for more properties. This tells us that $\rho = (n - 1, 1)$ gives the largest eigenvalue not equal to $1$.
\end{remark}



The following lemma is well-known in the literature, e.g.\ \cite{kostka}. 

\begin{lemma}
Let $\lambda, \lambda', \rho$ be partitions of $n$. Then,
\begin{enumerate}[(a)]
    \item $m_\rho^\lambda \neq 0$ if and only if $\lambda \prec \rho$.
    \item If $\lambda \prec \lambda' \prec \rho$, then $m_\rho^\lambda \ge m_\rho^{\lambda'}$.
\end{enumerate}
\end{lemma}

\begin{remark}
No such monotonicity exists in $\rho$. For example, with $n = 4$, $\lambda = (1, 1, 1, 1)$, we have $(1, 1, 1, 1) \prec (2, 1, 1) \prec (2, 2)$ yet $m_{(1, 1, 1, 1)}^\lambda = 1$, $m_{(2, 1, 1)}^\lambda = 3$, $m_{(2, 2)}^\lambda = 2$.

\end{remark}

\subsection{Average Case Mixing Time}

For $2 \times 2$ contingency tables, Corollary \ref{cor: multCases} suffices to determine the multiplicities of all eigenvalues. 

 \begin{lemma} \label{lem: avg}
For $\lambda = (n - k, k), \mu = (n - \ell, \ell)$ for $k\le \ell \le \lfloor n/2 \rfloor$ if $t = (n/4)(\log(k) + c)$ for $c > 0$ then
\[
\sum_{\mx \in \T_{\lambda, \mu}} \pi_{\lambda, \mu}(\mx) \|P^t_\mx - \pi_{\lambda, \mu} \|_{TV}^2 \le e^{-c}.
\]
If $t = cn/4$, then 
\[
\sum_{\mx \in \T_{\lambda, \mu}} \pi_{\lambda, \mu}(\mx) \chi^2_\mx(t) \ge 1 - c.
\]
\end{lemma}

\begin{proof}
Suppose $\rho = (n - m, m)$ with $m \le k \le \ell$. By Corollary \ref{cor: multCases}(c) $m_\rho^{\lambda} = m_\rho^\mu = 1$, so the multiplicity of $\beta_\rho = 1 - 2m(n + 1 -m)/n^2$ in the transpositions chain on $\T_{\lambda, \mu}$ is $1$. For any other partition $\rho$, the multiplicity is $0$. Thus,
\begin{align*}
\sum_{\mx \in \T_{\lambda, \mu}} \pi(\mx) \chi^2_\mx(t) &= \sum_{\rho \neq (n)} \beta_\rho^{2t} m_\rho^2 \\
&= \sum_{m = 1}^{k} \left( 1 - \frac{2m(n + 1 - m)}{n^2} \right)^{2t} \\
&\le \sum_{m = 1}^{k} \exp \left( -2t \frac{2m(n + 1 - m)}{n^2} \right) \le \exp \left( - \frac{4t}{n} + \log(k) \right),
\end{align*}
which is $\le e^{-c}$ for $t = (n/4)(\log(k) + c)$.

The lower bound comes from using the term in the sum with $\rho = (n - 1, 1)$ (which gives the largest eigenvalue, with multiplicity $1$):
\begin{align*}
\sum_{\mx \in \T_{\lambda, \mu}} \pi(\mx) \chi^2_\mx(t) &\ge \beta_{(n-1, 1)}^{2t} m_{(n-1, 1)}^2 \\
&=  \left( 1 - \frac{2}{n} \right)^{2t} \\
&\ge 1 - \frac{4t}{n} = 1 - c,
\end{align*}
if $t = cn/4$. The last inequality is Bernoulli's inequality.
\end{proof}

\begin{remark}
A table in $\T_{(n-k, k), (n -\ell, \ell)}$ is determined by one entry, say the $(2, 2)$ entry $X$, so the table is
\[
\begin{pmatrix} n - k - \ell + X & \ell - X \cr
k - X & X 
\end{pmatrix}.
\]
Then $X \in \{0, 1, \dots, k \}$ has the uni-variate hypergeometric distribution with parameters $n, k$. The transitions for $X$ are
\begin{align*}
&P(a, a + 1) = \frac{2(k - a)(\ell - a)}{n^2}, \quad a < k \\
&P(a, a - 1) = \frac{2a(n -k - \ell + a)}{n^2},  \quad a > 0.
\end{align*}
\end{remark}

\begin{remark}
In contrast to the previous remark, the Bernoulli-Laplace urn has transitions
\begin{align*}
&P(a, a + 1) = \frac{2(k - a)(\ell - a)}{k(n-k)}, \quad a < k \\
&P(a, a - 1) = \frac{2a(n - k - \ell + a)}{k(n-k)},  \quad a > 0.
\end{align*}
In \cite{diaconisShahshahani1981}, for the special case $k = \ell = n/2$ (which corresponds to the state space $\T_{(n/2, n/2), (n/2, n/2)}$), the Bernoulli-Laplace chain is proven to mix $(n/8) \log(n)$ steps. The paper \cite{nestoridi} studies a more general model in which $k$ balls are swapped between the two urns; the chain is shown to have mixing time $(n/4k)\log(n)$. These are special cases of the random walk on a distance regular graph studied in \cite{belsley}.
\end{remark}

\begin{remark}
Note that the random transpositions Markov chain on $S_n$ is transitive, and so the behavior of the chain, especially the mixing time, does not depend on the starting state. The chain lumped to $\T_{\lambda, \mu}$ is not transitive, and the mixing time could heavily rely on the starting state. It is expected that the mixing time is fastest starting from states with high probability (large double cosets) and slowest from the states with low probability (small double cosets). From Remark \ref{remark: largest}, the highest probability table is the one closest to the `independence table'; the tables with low probability are the sparse tables with many zeros. 
\end{remark}

\begin{remark}
For any $\lambda, \mu$, suppose $T \in \T_{\lambda, \mu}$. From the table $T$ we can create a $2 \times 2$ table by `collapsing' columns $2:J$ and rows $2:I$. That is, set
\begin{align*}
    &\widetilde{T}(1, 1) = T(1, 1) \\
    &\widetilde{T}(1, 2) = \sum_{j = 2}^{J} T(1, j) \\
    &\widetilde{T}(2, 1) = \sum_{i = 2}^{I} T(i, 1) \\
    &\widetilde{T}(2, 2) = \lambda_2 - \widetilde{T}(2, 1)  = \mu_2 - \widetilde{T}(1, 2).
\end{align*}
Then $\widetilde{T} \in \T_{(\lambda_1, n - \lambda_1), (\mu_1, n - \mu_1)}$. Futhermore,  $T(1, 1)$ has the uni-variate hypergeometric distribution with parameters $\mu_1, \lambda_1$. The random transpositions chain on $\T_{\lambda, \mu}$ lumps to the uni-variate Markov chain on  $\T_{(\lambda_1, n - \lambda_1), (\mu_1, n - \mu_1)}$. Thus the mixing time for the $(1, 1)$-entry Markov chain is a lower bound for the mixing time of the full process.
\end{remark}


\subsection{Orthogonal Polynomials} \label{sec: poly}

This section proves part (3) of Theorem \ref{thm: ev}, that the eigenfunctions of the random transpositions chain on contingency tables are orthogonal polynomials. While it is difficult to find an explicit formula for all of these polynomials, analysis of the chain provides a way to calculate the linear and quadratic polynomials. 


\begin{theorem}
Let $\lambda = (\lambda_1, \dots, \lambda_I), \mu = (\mu_1, \dots, \mu_J)$ be partitions of $n$ and $\{T_t \}_{t \ge 0}$ the random transpositions Markov chain on the space of contingency tables $\mathcal{T}_{\lambda, \mu}$ with transition matrix $P$. For any $m$ and $\mx \in \T_{\lambda, \mu}$,
\begin{align*}
\E[T_{t+1}(i, j)^m \mid T_t = \mx] &= x_{ij}^m \left( 1 - \frac{2m(n + 1 - m)}{n^2} \right) + (\text{terms in x of degree} \,\,\, < m ).
\end{align*}
The eigenfunctions for $P$ are polynomials.
\end{theorem}

\begin{proof}
Let $\{T_t \}_{t \ge 0}$ represent the Markov chain on $\T_{\lambda, \mu}$, with $T_t(i, j)$ denoting the value in the $(i, j)$ cell. Let $\mx \in \T_{\lambda, \mu}$. Marginally, each entry of the table is a birth-death process: For any $1 \le i \le I, 1 \le j \le J$,
\begin{align}
& \P\left( T_{t+1}(i, j) = x_{ij} + 1 \mid T_t = \mx \right) = \sum_{\ell \neq j} \sum_{k \neq i} \frac{2 \cdot  x_{i \ell} x_{kj}}{n^2} = \frac{2(\lambda_i - x_{ij})(\mu_j - x_{ij})}{n^2} \label{eqn: t1}\\
& \P\left( T_{t+1}(i, j) = x_{ij} - 1 \mid T_t = \mx \right) = \sum_{\ell \neq j} \sum_{k \neq i} \frac{2 x_{ij} x_{kl}}{n^2} = \frac{2x_{ij}(n - \lambda_i - \mu_j + x_{ij})}{n^2}. \label{eqn: t2}
\end{align}
These transitions allow for the calculation of $\E[T_{t+1}(i, j)^m \mid T_t = \mx]$, for any integer power $m$. That is,
\begin{align*}
\E[T_{t+1}(i, j)^m \mid T_t = \mx] &= (x_{ij} + 1)^m \cdot \frac{2(\lambda_i - x_{ij})(\mu_j - x_{ij})}{n^2} + (x_{ij} - 1)^m \cdot \frac{2x_{ij}(n - \lambda_i - \mu_j + x_{ij})}{n^2} \\
& \,\,\,\,\,\,\,\,\,\, + x_{ij}^m \left( 1 - \frac{2x_{ij}(n - \lambda_i - \mu_j + x_{ij})}{n^2} - \frac{2(\lambda_i - x_{ij})(\mu_j - x_{ij})}{n^2} \right) \\
&= x_{ij}^m + \left( m x_{ij}^{m-1} + \frac{m(m -1)}{2} x_{ij}^{m -2} + \sum_{\ell = 3}^m \binom{m}{\ell} x_{ij}^{m - \ell} \right) \cdot \frac{2(\lambda_i - x_{ij})(\mu_j - x_{ij})}{n^2} \\
& \,\,\,\,\,\,\,\,\,\, + \left(-m x_{ij}^{m-1} + \frac{m(m - 1)}{2} x_{ij}^{m-2}  + \sum_{\ell = 3}^m \binom{m}{\ell} x_{ij}^{m - \ell} (-1)^{\ell} \right) \cdot \frac{2x_{ij}(n - \lambda_i - \mu_j + x_{ij})}{n^2} \\
&= x_{ij}^m \left( 1 - \frac{2m(n + 1 - m)}{n^2} \right) + (\text{terms in x of degree} \,\,\, < m ).
\end{align*}
By Lemma \ref{lem: khare}, this condition means the eigenfunctions for $P$ are polynomials. 
\end{proof}

Straightforward calculation of the transition probabilities gives the linear eigenfunctions:

\begin{lemma} \label{lem: linearEF}
For any $\lambda, \mu$ and $1 \le i \le I, 1 \le j \le J$, the functions
\[
f_{ij}(\mx) := x_{ij} - \frac{\lambda_i \mu_j}{n}
\]
are eigenvectors with eigenvalue $1 - 2/n$. 
\end{lemma}

\begin{proof}
For $\mx \in \T_{\lambda, \mu}$, the expected value of one entry of the table after one step of the chain can be computed using \eqref{eqn: t1} and \eqref{eqn: t2}:
\begin{align}
\E[T_1(i, j) \mid T_0 = \mx] &= (x_{ij} + 1)\cdot \frac{2(\lambda_i - x_{ij})(\mu_j - x_{ij})}{n^2} + (x_{ij} - 1) \cdot \frac{2x_{ij}(n - \lambda_i - \mu_j + x_{ij})}{n^2} \notag \\ 
& \,\,\,\,\,\,\,\,\,\, + x_{ij} \left( 1 - \frac{2x_{ij}(n - \lambda_i - \mu_j + x_{ij})}{n^2} - \frac{2(\lambda_i - x_{ij})(\mu_j - x_{ij})}{n^2} \right) \notag \\ 
&= \frac{2(\lambda_i - x_{ij})(\mu_j - x_{ij})}{n^2} - \frac{2x_{ij}(n - \lambda_i - \mu_j + x_{ij})}{n^2} + x_{ij} \notag \\ 
&= x_{ij} \left( 1 - \frac{2}{n} \right) + \frac{2 \lambda_i \mu_j}{n^2}. \label{eqn: linExpectation}
\end{align}
Thus,
\begin{align*}
    \E[f_{ij}(T_1) \mid T_0 = \mx] &= x_{ij} \left( 1 - \frac{2}{n} \right) + \frac{2 \lambda_i \mu_j}{n^2} - \frac{\lambda_i \mu_j}{n} \\
    &= \left(1 - \frac{2}{n}\right)\left(x_{ij} - \frac{\lambda_i \mu_j}{n} \right) = \left(1 - \frac{2}{n}\right) \cdot f_{ij}(\mx).
\end{align*}
\end{proof}

\begin{remark}
The set $\{f_{ij} \}_{1 \le i \le I -1, 1 \le j \le J-1 }$ of eigenvectors from Lemma \ref{lem: linearEF} is a basis for the eigenspace with eigenvalue $1 - 2/n$ (as Corollary \ref{cor: multCases} (a) showed the multiplicity of $1-2/n$ is $(I-1)(J-1)$). The current version of the functions $f_{ij}$ are \emph{not} orthogonal with respect to $\pi_{\lambda, \mu}$, because
\begin{align} \label{eqn: crossMoments}
    \E_{\pi_{\lambda, \mu}}[T_{ij}T_{kl}] = \begin{cases}
    \frac{\lambda_i \mu_j \lambda_k \mu_l}{n(n-1)} & \text{if} \quad i \neq k, j \neq l \\
    \frac{\lambda_i \mu_j \lambda_k (\mu_j - 1)}{n(n-1)} & \text{if} \quad i \neq k, j = l \\
    \frac{\lambda_i^2 \mu_j^2}{n^2} + \frac{\lambda_i \mu_j(n - \lambda_i)(n - \mu_j)}{n^2(n-1)} & \text{if} \quad i = j, k = l
    \end{cases}.
\end{align}
If $i \neq k, j \neq l$, then
\begin{align*}
    \E_{\pi_{\lambda, \mu}}[f_{i, j}(T)f_{k, l}(T)] &= \frac{\lambda_i \mu_j \lambda_k \mu_l}{n(n-1)} - 2 \frac{\lambda_i \mu_j \lambda_k \mu_l}{n^2} + \frac{\lambda_i \mu_j \lambda_k \mu_l}{n^2} \\
    &= \frac{\lambda_i \mu_j \lambda_k \mu_l}{n} \left(  \frac{1}{n(n-1)} \right).
\end{align*}
\end{remark}

To find the quadratic eigenvectors, the following computations are needed. 
\begin{lemma} \label{lem: secondMoments}
For any $\lambda, \mu$, $1 \le i, k \le I, 1 \le j, l \le J$ and $\mx \in \T_{\lambda, \mu}$,
\begin{equation}
\E\left[T_1(i, j)T_1(k, l) \mid T_0 = \mx \right] = \begin{cases} x_{ij}x_{kl} \left(1 - \frac{4}{n} + \frac{2}{n^2} \right)  + \frac{2}{n^2} \left( x_{kl} \lambda_i \mu_j + x_{ij} \lambda_k \mu_l + x_{il} x_{kj}  
\right) & i \neq k, j \neq l \\
x_{ij}x_{kj} \left(1 - \frac{4}{n} + \frac{4}{n^2} \right)  + \frac{2}{n^2} \left( x_{kj} ( \lambda_i \mu_j - \lambda_i) + x_{ij} (\lambda_k \mu_j  - \lambda_k) \right) & i \neq k, j = l \\
x_{ij}^2\left( 1 - \frac{4}{n} + \frac{4}{n^2} \right) + \frac{2}{n^2} \left( x_{ij}(2 \lambda_i \mu_j - 2 \lambda_i - 2 \mu_j + n ) + \lambda_i \mu_j \right) & i = k, j = l 
\end{cases}.
\end{equation}
\end{lemma}

\begin{proof}
For the first case, if $i \neq k$ and $j \neq l$ then the only situation in which both the $(i, j)$ and $(k, l)$ entries are chosen is if they are opposite corners of the box chosen for the swap move. 
In the following calculation, \eqref{eqn: line1} is the case where both the $(i, j)$ and $(k, l)$ entries change, \eqref{eqn: line2} is the case where only the $(i, j)$ entry changes, and \eqref{eqn: line3} is the case where only the $(k, l)$ entry changes.
\begin{align}
    \E \left[ T_1(i, j)T_1(k, l) \right. & \left. - x_{ij} x_{kl}  \mid T_0 = \mx \right] = (-x_{ij} - x_{kl} + 1) \frac{2x_{ij}x_{kl}}{n^2} + (x_{ij} + x_{kl} + 1) \frac{2x_{il}x_{kj}}{n^2} \label{eqn: line1}\\
    &-x_{kl} \frac{2x_{ij}(n - \lambda_i - \mu_j + x_{ij} - x_{kl})}{n^2} 
    + x_{kl} \frac{2((\lambda_i - x_{ij})(\mu_j - x_{ij}) - x_{il}x_{kj})}{n^2} \label{eqn: line2} \\
    &- x_{ij}\frac{2x_{kl}(n - \lambda_k - \mu_l + x_{kl} - x_{ij})}{n^2} + x_{ij}\frac{2((\lambda_k - x_{kl})(\mu_l - x_{kl}) - x_{il}x_{kj})}{n^2} \label{eqn: line3} \\
    &=\frac{2}{n^2} \left( x_{ij} x_{kl} - 2n x_{ij} x_{kl} + x_{ij} \lambda_k \mu_l + x_{kl} \lambda_i \mu_j + x_{il} x_{kj} \right) \notag 
\end{align}

Now suppose $i \neq k$ and $j = l$. In this case the $(i, j)$ and $(k, l)$ entries of the table could only both change if one increases and the other decreases. This gives
\begin{align}
    \E \left[ T_1(i, j)T_1(k, j) \right. & \left. - x_{ij} x_{kj}  \mid T_0 = \mx \right] = (x_{ij} - x_{kj} - 1) \frac{2x_{ij}(\lambda_k - x_{kj})}{n^2} + (-x_{ij} + x_{kj} - 1) \frac{2x_{kj}(\lambda_i - x_{ij})}{n^2} \\
    &-x_{kj} \frac{2x_{ij}(n - \lambda_i - \lambda_k - \mu_j + x_{ij} + x_{kj})}{n^2} 
    + x_{kj} \frac{2(\lambda_i - x_{ij})(\mu_j - x_{ij} - x_{kj})}{n^2}  \\
    &- x_{ij}\frac{2x_{kj}(n - \lambda_k - \lambda_i - \mu_j + x_{kj} + x_{ij})}{n^2} + x_{ij}\frac{2(\lambda_k - x_{kj})(\mu_j - x_{kj} - x_{ij})}{n^2} \\
    &= \frac{2}{n^2} \left(2 x_{ij}x_{kj} - 2n x_{ij}x_{kj} - \lambda_k x_{ij} - \lambda_i x_{kj} + x_{kj} \lambda_i \mu_j + x_{ij} \lambda_k \mu_j \right). \notag
\end{align}

Finally for the case $i = k, j = l$, this is the second moment of a birth death process. Using the transitions \eqref{eqn: t1} and \eqref{eqn: t2}, the calculation is
\begin{align*}
\E[T_1(i, j)^2 \mid T_0 = \mx ] &= (x_{ij} + 1 )^2 \left( \frac{2(\lambda_i - x_{ij})(\mu_j - x_{ij})}{n^2} \right) + (x_{ij} - 1)^2 \left( \frac{2x_{ij}(n - \lambda_i - \mu_j + x_{ij})}{n^2} \right) \\
& \,\,\,\,\,\,\,\,\,\,\, + x_{ij}^2 \P(T_1(i, j) = x_{ij} \mid T_0 = \mx) \\
&= x_{ij}^2\left( 1 - \frac{4}{n} + \frac{4}{n^2} \right) + \frac{2}{n^2} \left( x_{ij}(2 \lambda_i \mu_j - 2 \lambda_i - 2 \mu_j + n ) + \lambda_i \mu_j \right).
\end{align*}
\end{proof}

\begin{lemma}[Quadratic Eigenfunctions]
Let $\lambda, \mu$ be partitions of $n$ with $|\lambda| = I, |\mu| = J$. Let $1 \le i, k \le I, 1 \le j, l \le J$ with $i \neq k, j \neq l$. For the Markov chain on $\T_{\lambda, \mu}$, the following functions are eigenfunctions with eigenvalue $1 - 4/n + 4/n^2$, defined for $\mx \in \T_{\lambda, \mu}$:
\begin{enumerate}[(a)]
    \item \begin{align*}
f_{(i,j), (k, l)}(\mx) &:= x_{ij}x_{kl}  - x_{ij} \frac{\lambda_k \mu_l}{n - 2} - x_{kl} \frac{\lambda_i \mu_j}{n - 2}  \\
&\,\,\,\,\,\,\, + x_{il}x_{kj} - x_{il} \frac{\lambda_k \mu_j}{n - 2} - x_{kj} \frac{\lambda_i \mu_l}{n - 2} + \frac{2\lambda_k \mu_l \lambda_i \mu_j}{(n-1)(n-2)}.
\end{align*}

\item \begin{align*}
    f_{(i, j), (k, j)}(\mx) &= x_{ij} x_{kj} - x_{ij} \frac{\lambda_i(\mu_j - 1)}{n - 2} - x_{kj} \frac{\lambda_k(\mu_j - 1)}{n-2} + \frac{\lambda_i \lambda_k \mu_j( \mu_j - 1)}{(n-1)(n-2)}.
\end{align*}

\item 
\begin{align*}
     f_{(i, j), (i, j)}(\mx)  &= x_{ij}^2 - x_{ij} \frac{2 \lambda_i \mu_j - 2 \lambda_i - 2 \mu_j + n}{n - 2} + \frac{\lambda_i \mu_j(1 + \lambda_i \mu_j - \lambda_i - \mu_j)}{(n - 1)(n-2)}.
\end{align*}
\end{enumerate}

\end{lemma}

\begin{remark}
Using Equation \ref{eqn: crossMoments}, one can check that $\E_{\pi_{\lambda, \mu}}[f_{(i, j), (k, l)}(T)] = 0$, for any $i, j, k, l$.
\end{remark}

\begin{proof}
The results follow from straightforward calculations using Lemma \ref{lem: secondMoments} and Equation \ref{eqn: linExpectation}. As an illustration, (a) is computed:  In $\E[f_{(i, j), (k, l)}(T_1) \mid T_0 = \mx]$, the degree $2$ terms will be
\begin{align*}
    x_{ij}x_{kl} \left(1 - \frac{4}{n} + \frac{2}{n^2} \right) + \frac{2}{n^2} x_{il} x_{kj} + x_{il}x_{kj} \left(1 - \frac{4}{n} + \frac{2}{n^2} \right) + \frac{2}{n^2} x_{ij} x_kl =  \left(1 - \frac{4}{n} + \frac{4}{n^2} \right) \left( x_{ij}x_{kl} + x_{il} x_{kj} \right).
\end{align*}
Degree $1$ terms arise from the expectation of six of the terms in $f_{(i, j), (k,l)}(T_1)$:
\begin{align*}
    &\E[T_1(i, j)T_1(k, l) \mid T_0 = \mx] \to \frac{2}{n^2} \left( x_{kl} \lambda_i \mu_j + x_{ij} \lambda_k \mu_l \right) \\
    &\E[T_1(i, l)T_1(k, j) \mid T_0 = \mx] \to \frac{2}{n^2} \left( x_{kj} \lambda_i \mu_l + x_{il} \lambda_k \mu_j \right) \\
    &\E\left[- T_1(i, j) \frac{\lambda_k \mu_l}{n - 2} \mid T_0 = \mx \right] \to - \left( 1 - \frac{2}{n}  \right) x_{ij} \frac{\lambda_k \mu_l}{n - 2}  \\
    &\E\left[- T_1(k, l) \frac{\lambda_i \mu_j}{n - 2} \mid T_0 = \mx \right] \to - \left( 1 - \frac{2}{n}  \right) x_{kl} \frac{\lambda_i \mu_j}{n - 2} \\
    &\E\left[- T_1(i, l) \frac{\lambda_k \mu_j}{n - 2} \mid T_0 = \mx \right] \to - \left( 1 - \frac{2}{n}  \right) x_{il} \frac{\lambda_k \mu_j}{n - 2} \\
    &\E\left[- T_1(k, j) \frac{\lambda_i \mu_l}{n - 2} \mid T_0 = \mx \right] \to - \left( 1 - \frac{2}{n}  \right) x_{kj} \frac{\lambda_i \mu_l}{n - 2} 
\end{align*}
Collecting the $x_{ij}$ terms gives
\begin{align*}
    x_{ij} \left( \frac{2}{n^2} \lambda_k \mu_l - \frac{1}{n} \lambda_k \mu_l \right) &= (n-2) \left( \frac{2}{n^2} - \frac{1}{n} \right) x_{ij} \frac{\lambda_k \mu_l}{n - 2} \\
    &=- \left( 1 - \frac{4}{n} + \frac{4}{n^2}\right)x_{ij} \frac{\lambda_k \mu_l}{n - 2},
\end{align*}
and the computation is the same for the other linear terms. Finally, the constant terms arise from:
\begin{align*}
    &\E\left[- T_1(i, j) \frac{\lambda_k \mu_l}{n - 2} \mid T_0 = \mx \right] \to - \frac{2 \lambda_i \mu_j}{n^2} \cdot \frac{\lambda_k \mu_l}{n - 2}  \\
    &\E\left[- T_1(k, l) \frac{\lambda_i \mu_j}{n - 2} \mid T_0 = \mx \right] \to - \frac{2 \lambda_k \mu_l}{n^2} \cdot \frac{\lambda_i \mu_j}{n - 2} \\
    &\E\left[- T_1(i, l) \frac{\lambda_k \mu_j}{n - 2} \mid T_0 = \mx \right] \to - \frac{2 \lambda_i \mu_l}{n^2} \cdot \frac{\lambda_k \mu_j}{n - 2} \\
    &\E\left[- T_1(k, j) \frac{\lambda_i \mu_l}{n - 2} \mid T_0 = \mx \right] \to - \frac{2 \lambda_k \mu_j}{n^2} \cdot  \frac{\lambda_i \mu_l}{n - 2}. 
\end{align*}
This gives
\begin{align*}
    \frac{2\lambda_k \mu_l \lambda_i \mu_j}{(n-1)(n-2)} - \frac{8 \lambda_k \mu_l \lambda_i \mu_j}{n^2(n-2)} = \left( 1 - \frac{4}{n} + \frac{4}{n^2}\right)\frac{2\lambda_k \mu_l \lambda_i \mu_j}{(n-1)(n-2)}.
\end{align*}
Further details are omitted. 
\end{proof}

\section{Mixing Time} \label{sec: mixingTime}

This section contains results on the mixing time for special cases. Throughout, keep in mind that random transpositions on $S_n$ has mixing time $(n/2)\log(n)$, which is an upper bound for the mixing time on $\T_{\lambda, \mu}$. The question then is: Does the lumping speed up mixing, and if so by what order? 

\subsection{Upper Bound for $2 \times J$ tables}

Due to the complicated nature of the orthogonal polynomials, the upper bound can only be analyzed for the specific starting states for which the kernel polynomials simplify. Suppose $\lambda = (n - k, k), \mu = (\mu_1, \dots, \mu_J)$, for $k \le \lfloor n/2 \rfloor$, and at least one $1 \le j \le J$ such that $\mu_j > k$.

 Let $k\mathbf{e}_j$ be the table with the second row all $0$ except $k$ in the $j$th column. For example, when $n = 10, \mu = (4, 3, 3), \lambda = (8, 2)$, we have the following table
\[
2\mathbf{e}_1 = \begin{pmatrix}
2 & 3 & 3 \\
2 & 0 & 0 
\end{pmatrix}.
\]
Note that the assumption $\mu_j > k$ ensure that this table exists.

Recall Proposition \ref{prop: kernel} which says, evaluated at these states, the kernel polynomials for the multivariate orthogonal polynomials have a simple closed-form expression:
\begin{equation} \label{eqn: kernel}
h_m(k \mathbf{e}_j, k \mathbf{e}_j)  = \binom{k}{m} \frac{(n - 2m + 1)n_{[m - 1]}(n - \mu_j)_{[m]}}{(n - k)_{[m]} (\mu_j)_{[m]}}.
\end{equation}
Recall the notation $a_{[m]} = a(a - 1) \hdots (a - m + 1)$  for the decreasing factorial, with $a_{[0]} := 1$.

\begin{theorem} \label{lem: 2Jupper}
Let $P$ be the transition matrix for the swap Markov chain $\T_{(n - k, k), \mu}$, with $k \le \lfloor n/2 \rfloor$ and $\mu = (\mu_1, \dots, \mu_J)$ with $\mu_j > k$ for at least one index $1 \le j \le J$. For any $c > 0$ and $1 \le j \le J$ such that $\mu_j > k$, 
\begin{enumerate}[(a)]
    \item If 
\[
t = \left( \frac{n}{4} + \frac{k(k-1)}{2(n - 2k)} \right) \cdot \left( \log\left( \frac{k \cdot n \cdot (n - \mu_j)}{(n - 2k)\cdot(\mu_j - k)} \right) + c \right),
\]
then $\chi^2_{k \mathbf{e}_j}(t)  \le e^{-c}$.

\item If 
\[
t = \frac{n}{8}\cdot \left( \log\left( \frac{k \cdot (n-1) \cdot (n - \mu_j)}{(n - k)\cdot \mu_j} \right) - c \right),
\]
then $\chi^2_{k \mathbf{e}_j}(t)  \ge e^c$.
\end{enumerate}
\end{theorem}

\begin{remark}
In \cite{khare} the kernel polynomials for the multivariate hypergeometric distribution are similarly used to analyze three classes of Bernoulli-Laplace type Markov chains, which have multivariate hypergeometric stationary distribution. Theorem \ref{lem: 2Jupper} can be compared to Proposition 4.2.1 \cite{khare}. The random transpositions chain on $\T_{(n - k, k), \mu}$ is essentially a special case of the Bernoulli Laplace level model. 


\end{remark}

\begin{proof}
Using the expression \eqref{eqn: kernel} for the kernel polynomials, the chi-squared distance is
\begin{align*}
\chi^2_{k \mathbf{e}_j}(t) &= \sum_{m = 1}^k \left(1 - \frac{2m(n + 1 - m)}{n^2} \right)^{2t} \cdot \binom{k}{m} \frac{(n - 2m + 1)n_{[m - 1]}(n - \mu_j)_{[m]}}{(n - k)_{[m]} (\mu_j)_{[m]}}.
\end{align*}
To help bound this sum, let $S_m$ be the $m$th term in the summand, i.e.\ 
\[
S_m = \left(1 - \frac{2m(n + 1 - m)}{n^2} \right)^{2t} \cdot \binom{k}{m} \frac{(n - 2m + 1)n_{[m - 1]}(n - \mu_j)_{[m]}}{(n - k)_{[m]} (\mu_j)_{[m]}}.
\]
Then the ratio of consecutive terms can be bounded:
\begin{align*}
\frac{S_{m+1}}{S_m} \le &\left( 1 - \frac{2(n - 2m)}{n^2 - 2m(n  + 1 - m)} \right)^{2t} \cdot \frac{k-m}{m + 1} \cdot \frac{n  - 2m - 1}{n - 2m + 1} \cdot \frac{(n - m + 1)(n - \mu_j -m)}{(n - k -m)(\mu_j - m)} \\
&\,\,\,\,\,\,\,\,\,\, \le \left(1 - \frac{2(n - 2k)}{n^2 - 2k(n + 1 - k)}  \right)^{2t} \cdot \frac{k}{2} \cdot \frac{n}{n -2k} \cdot \frac{n - \mu_j}{\mu_j - k} \\
&\,\,\,\,\,\,\,\,\,\, \le \frac{k}{2} \cdot \frac{n}{n -2k} \cdot \frac{n - \mu_j}{\mu_j - k} \exp \left( - 2t \left( \frac{2(n - 2k)}{n^2 - 2k(n + 1 - k)} \right) \right).
\end{align*}
If $c > 0$ and 
\[
t = \left( \frac{n^2 - 2k(n + 1 - k)}{4(n - 2k)} \right)  \cdot \left( \log\left( k \cdot \frac{n}{n - 2k} \cdot \frac{n - \mu_j}{\mu_j - k} \right) + c \right),
\]
then the ratio $S_{m+1}/S_m$ is less than $e^{-c}/2 \le 1/2$. Also the first term $S_1$ is the largest term, and
\[
S_1 = \left(1 - \frac{2}{n} \right)^{2t} \cdot \frac{k(n - 1)(n - \mu_j)}{(n - k) (\mu_j)} \le e^{-c}.
\]
Thus, 
\begin{align*}
\sum_{m = 1}^k S_m  \le S_1 \sum_{m = 0}^\infty \frac{1}{2^m} \le e^{-c}.
\end{align*}

The bound for part (b) comes from considering the contribution to the sum from the largest eigenvalues, which is for $m = 1$. This gives
\begin{align*}
    \chi^2_{k \mathbf{e}_j}(t) &\ge  \left(1 - \frac{2}{n} \right)^{2t} \cdot k \cdot \frac{(n - 1)n_{[0]}(n - \mu_j)_{[1]}}{(n - k)_{[1]} (\mu_j)_{[1]}} = \left(1 - \frac{2}{n} \right)^{2t} \cdot k\cdot \frac{(n - 1)\cdot (n - \mu_j)}{(n - k) \cdot \mu_j} \\
    &= \exp \left( 2t \log \left( 1 - \frac{2}{n} \right) + \log \left( k\cdot \frac{(n - 1)\cdot (n - \mu_j)}{(n - k) \cdot \mu_j} \right)\right) \\
    &\ge \exp \left( 2t  \left( - \frac{4}{n}  \right) + \log \left( k\cdot \frac{(n - 1)\cdot (n - \mu_j)}{(n - k) \cdot \mu_j} \right)\right) ,
\end{align*}
using that $\log(1 - x) \ge - x - x^2 > - 2x$ for $0 \le x \le 1/2$. The sum is then $\ge e^{c}$ for $t = (n/8)\left( \log \left( k(n-1)(n- \mu_j)/((n-k) \mu_j) \right) - c\right)$.
\end{proof}

\begin{example} \label{example: upper22}
Suppose $\lambda = (n - k, k), \mu = (n - \ell, \ell)$ with $\ell > k$. Then Lemma \ref{lem: 2Jupper} gives $\chi^2_{k \mathbf{e}_2}(t) \le e^{-c}$ for
\[
t = \left( \frac{n}{4} + \frac{k(k-1)}{2(n - 2k)} \right) \cdot \left( \log\left( \frac{k \cdot n \cdot (n - \ell)}{(n - 2k)\cdot(\ell - k)} \right) + c \right).
\]
In particular if $n$ is even, $k = n/2 - 1$, $\ell = n/2$, then
\[
t = \left( \frac{n}{4} + \frac{(n/2-1)(n/2-2)}{2} \right) \cdot \left( \log\left( (n/2 - 1) \cdot n \cdot (n/2)\right) + c \right) \sim \left( \frac{n}{4} + \frac{n^2}{8} \right) \left(\log(n^3) + c \right).
\]
Note that this is a factor of $n$ larger than the bound in Lemma \ref{lem: avg} for the distance averaged over all starting states. This could be due to slower mixing from the extreme starting state:
\[
k \mathbf{e}_2 = \begin{pmatrix}
n - \ell + k & \ell - k \\
0 & k
\end{pmatrix},
\]
which has small probability under $\pi_{(n -k , k), (n - \ell, \ell)}$.

\end{example}

\subsection{Lower Bound}

Because the linear eigenfunctions are known for any size contingency table, Wilson's method works to give a general lower bound. 

\begin{theorem}[Wilson's Method, Theorem 13.5 in \cite{levinPeres}]
Let $(X_t)_{t \ge 0}$ be an irreducible aperiodic Markov chain with state space $\Omega$ and transition matrix $P$. Let $\phi$ be an eigenfunction of $P$ with eigenvalue $\lambda$ with $1/2 < \lambda < 1$. Fix $0 < \epsilon < 1$ and let $R > 0$ satisfy
\[
\E \left[ |\phi(X_1) - \phi(x)|^2 \mid X_0 = x \right] \le R, \,\,\,\,\,\,\,\, \forall x \in \Omega.
\]
Then for any $x \in \Omega$,
\[
t_{mix}(\epsilon) \ge \frac{1}{2 \log(1/\lambda)} \left[ \log \left( \frac{(1 - \lambda) \phi(x)^2}{2 R} \right) + \log \left( \frac{1 - \epsilon}{\epsilon} \right) \right].
\]
\end{theorem}

For the contingency table Markov chain on $\T_{\lambda, \mu}$ the linear functions
\[
f_{ij}(\mx) = x_{ij} - \frac{\lambda_i \mu_j}{n}, \quad \mx \in \T_{\lambda, \mu}
\]
for any $i, j$, are eigenfunctions with eigenvalue $1 - 2/n$. These will be used in Wilson's method to get the following result. 

\begin{lemma} \label{lem: CTlowerBound}
Let $\lambda = (\lambda_1, \dots, \lambda_I), \mu = (\mu_1, \dots, \mu_J)$ be any partitions of $n$. For any $i, j$ and $c > 0$,
\[
t_{mix} \ge \begin{cases} \left( \frac{n}{4} - \frac{1}{2} \right) \left( \log \left( m_{ij} - \frac{\lambda_i \mu_j}{n} \right) - c \right) & \text{if} \,\,\, n \ge 2(\lambda_i + \mu_j) \\
\left( \frac{n}{4} - \frac{1}{2} \right) \left( \log \left( \frac{1}{2} \frac{(n m_{ij} - \lambda_i \mu_j)^2}{n(n+2) \lambda_i \mu_j}  \right) - c \right) & \text{if} \,\,\, n < 2(\lambda_i + \mu_j)
\end{cases},
\]
where $m_{ij} = \min(\lambda_i, \mu_j)$.
\end{lemma}

\begin{proof}
From Lemmas \ref{lem: linearEF} and \ref{lem: secondMoments}, 
\begin{align*}
&\E[T_{t+1}(i, j) \mid T_t = \mx] = x_{ij} \left( 1 - \frac{2}{n} \right) + \frac{2 \lambda_i \mu_j}{n^2} \\
&\E[T_{t+1}(i, j)^2 \mid T_t = \mx] = x_{ij}^2 \left( 1 - \frac{4}{n} + \frac{4}{n^2} \right) + \frac{2}{n^2} \left( x_{ij}(2 \lambda_i \mu_j - 2 \lambda_i - 2 \mu_j + n) + \lambda_i \mu_j) \right).
\end{align*}
This allows the calculation of $\E[|\phi(T_1) - \phi(\mx)|^2 \mid T_0 = \mx]$ for $\phi = f_{ij}$:
\begin{align*}
\E[|f_{ij}(T_1) - f_{ij}(\mx)|^2 \mid T_0 = \mx] &= \E[|T_1(i, j) - x_{ij}|^2 \mid T_0 = \mx] \\
&= x_{ij}^2 \left( 1 - \frac{4}{n} + \frac{4}{n^2} \right) + \frac{2}{n^2} \left( x_{ij}(2 \lambda_i \mu_j - 2 \lambda_i - 2 \mu_j + n)  + \lambda_i \mu_j) \right) \\
&\,\,\,\,\,\,\,\,\,\, - 2 x_{ij} \left(x_{ij} \left( 1 - \frac{2}{n} \right) + \frac{2 \lambda_i \mu_j}{n^2} \right) + x_{ij}^2 \\ 
&= x_{ij}^2 \cdot \frac{4}{n^2} + x_{ij} \cdot \frac{2}{n^2} \left(n - 2 \lambda_i - 2 \mu_j \right) + \frac{2 \lambda_i \mu_j}{n^2} =: A
\end{align*}
First suppose that $i, j$ are such that $n \ge  2(\lambda_i + \mu_j)$, and note that $x_{ij} \le \min(\lambda_i, \mu_j) =: m_{ij}$ for $\mx \in \T_{\lambda, \mu}$. A bound then is
\begin{align*}
A &\le \frac{4 m_{ij}^2}{n^2} + m_{ij} \frac{2}{n^2} \left( n - 2 \lambda_i - 2 \mu_j \right) + \frac{2 \lambda_i \mu_j}{n^2} \\
&= \frac{2 m_{ij}}{n} - \frac{2 \lambda_i \mu_j}{n^2}.
\end{align*}
This is the constant $R$ that can be applied to Wilson's method.
Using $\mx \in \T_{\lambda, \mu}$ such that $x_{ij} = m_{ij} = \min(\lambda_i, \mu_j)$:
\begin{align*}
\frac{(1 - \lambda) f_{ij}(\mx)^2}{2R} &= \left. \frac{2}{n} \left( x_{ij} - \frac{\lambda_i \mu_j}{n} \right)^2 \middle/  2 \left( \frac{2 m_{ij}}{n} - \frac{2 \lambda_i \mu_j}{n^2} \right) \right. \\
&= \left. \frac{2}{n} \left( m_{ij} - \frac{\lambda_i \mu_j}{n} \right)^2 \middle/  2 \left( \frac{2 m_{ij}}{n} - \frac{2 \lambda_i \mu_j}{n^2} \right) \right. = \frac{1}{2} \left( m_{ij} - \frac{\lambda_i \mu_j}{n} \right)
\end{align*}
If $n <  2(\lambda_i + \mu_j)$, then
\[
A \le \frac{4m_{ij}^2}{n^2} + \frac{2 \lambda_i \mu_j}{n} \le \frac{2 \lambda_i \mu_j}{n}\left(1 + \frac{2}{n} \right),
\]
and in this case
\begin{align*}
\frac{(1 - \lambda) f_{ij}(\mx)^2}{2R} = \left. \frac{2}{n} \left( m_{ij} - \frac{\lambda_i \mu_j}{n} \right)^2 \middle/  \frac{4 \lambda_i \mu_j}{n} \left( 1 + \frac{2}{n}  \right) \right. = \frac{1}{2} \frac{(n m_{ij} - \lambda_i \mu_j)^2}{n(n+2) \lambda_i \mu_j} 
\end{align*}
 Finally, note that, with $\lambda = 1 -2/n$ and using the bound $1/-\log(1 - \gamma) \ge 1/\gamma - 1$,
\[
\frac{1}{2\log(1/\lambda)} = \frac{1}{-2 \log(1 - 2/n)} \ge \frac{n}{4} - \frac{1}{2}.
\]

\end{proof}

\begin{example}
Suppose $\lambda = (n - k, k), \mu = (n - \ell, \ell)$ with $k < \ell \le n/2$ (as in Example \ref{example: upper22}). Then $2n - \ell - k > n$, so the second case in Lemma \ref{lem: CTlowerBound} always applies to the $(1, 1)$ entry of the table. Since $m_{11} = \min(n - \ell, n - k) = n - \ell$, a lower bound is
\begin{align*}
    \tm &\ge \left( \frac{n}{4} - \frac{1}{2} \right) \left( \log \left( \frac{1}{2} \frac{(n(n - \ell) - (n - \ell)(n - k) )^2}{n(n+2) (n - \ell)(n - k) }  \right) - c \right) \\
    &= \left( \frac{n}{4} - \frac{1}{2} \right) \left( \log \left( \frac{1}{2} \frac{(n - \ell)k^2}{n(n+2)(n - k)} \right) - c \right).
\end{align*}
For example, if $k = \ell = n/2$, then the expression inside the $\log$ is equal to $(n/2)^2/(2n(n+2)) \sim 1/8$.

If $\ell, k$ are small enough so that $k + \ell \le n/2$, then the first case of Lemma \ref{lem: CTlowerBound} applies to the $(2, 2)$ entry, with $m_{22} = \min(k, \ell) = k$. The lower bound is then 
\begin{align*}
    \tm = \left( \frac{n}{4} - \frac{1}{2} \right) \left( \log \left( k - \frac{k \ell }{n} \right) - c \right).
\end{align*}
For example, $\ell = k = n/4$, then the expression inside the $\log$ is equal to $(n^2/4 - n^2/16)/n \sim (3/16) n$. 
\end{example}

\section{Further Directions} \label{sec: future}

This section discussions some future directions and applications of the random transpositions chain on contingency tables. Section \ref{sec: dataAnalysis} explores how the linear and quadratic eigenfunctions of the Fisher-Yates distribution could be used in statistical applications, to decompose the chi-squared statistic. Section \ref{sec: multiway} notes how the random transpositions chain can be extended to multi-way tables (coming from data with more than $2$ categorical features). Section \ref{sec: metropolisCT} considers the random transpositions chain on $\T_{\lambda, \mu}$ transformed via the Metropolis-Hastings algorithm to a new Markov chain which has uniform stationary distribution (and vice-versa, the symmetric swap chain can be transformed to have Fisher-Yates stationary distribution). The relaxation times of the chains can be compared. The constants in the comparison depend on $\lambda, \mu$ and the size of the table, but the conclusion is that for sparse tables the metropolized version of random transpositions has a significantly smaller relaxation time than the symmetric swap chain. Section \ref{sec: GLnparabolic} notes the $q$-analog of the Fisher-Yates distribution arises via double cosets. 

\subsection{Data Analysis} \label{sec: dataAnalysis}

This section discusses potential statistical applications of the eigenfunctions for the random transpositions chain on contingency tables. First, some classic datasets are described.

\paragraph{Datasets.}
Tables \ref{tab: midtown}, \ref{tab: victoria}, \ref{tab: hair} are classical real data tables with large $\chi^2$ statistic. Figure \ref{fig: stats} shows a histogram of the the quadratic eigenfunctions evaluated on each of these tables, as well as a plot of the Pearson residuals. We have not succeded in finding any extra structure from these displays but believe that they may sometimes be informative.


\begin{table}
\begin{center}
\begin{tabular}{c | c c c c | c}
& Well & Mild & Moderate & Impaired & \textbf{Total} \\ \hline
A & 64 & 94 & 58 & 46 & 262 \\  
B & 57 & 94 & 54 & 40 & 245 \\ 
C & 57 & 105 & 65 & 60 & 287 \\ 
D &72 & 141 & 77 & 94 & 384 \\ 
E & 36 & 97 & 54 & 78 & 265 \\ 
F & 21 & 71 & 54 & 71 & 217 \\ \hline
\textbf{Total} & 307 & 602 & 362 & 389 & 1660
\end{tabular} 
\caption{Midtown Manhattan Mental Health Study data} \label{tab: midtown}
\end{center}
\end{table}

Table \ref{tab: midtown} shows data from an epidemiological survey known as the Midtown Manhattan Mental Health Study  \cite{midtown}. Rows record parent's socioeconomic status (ranging from A = high, to F = low) and columns severity of mental illness. The $\chi^2$ statistic is $45.98$ on $15$ degrees of freedom.


\begin{table}
\begin{center}
\begin{tabular}{c | c c c c c c c c c c c c | c}
 &\textbf{Jan} & \textbf{Feb} & \textbf{March} & \textbf{April}  & \textbf{May} & \textbf{June} & \textbf{July} & \textbf{Aug} & \textbf{Sep} & \textbf{Oct} & \textbf{Nov} & \textbf{Dec}  & \textbf{Total}\\ \hline
\textbf{Jan} & 1 & 0 & 0& 0& 1& 2& 0 & 0 & 1 & 0& 1& 0 & 6\\
   \textbf{Feb} &              1& 0& 0& 1& 0& 0& 0& 0& 0& 1& 0& 2 & 5 \\
 \textbf{March} &                1& 0& 0& 0& 2& 1& 0& 0& 0& 0& 0& 1 & 5 \\
\textbf{April}  &                 3& 0& 2& 0& 0& 0& 1& 0& 1& 3& 1& 1 & 12 \\
 \textbf{May} &                 2& 1& 1& 1& 1& 1& 1& 1& 1& 1& 1& 0 & 12 \\
 \textbf{June} &                2& 0& 0& 0& 1& 0& 0& 0& 0& 0& 0& 0 & 3 \\
 \textbf{July} &                2& 0& 2& 1& 0& 0& 0& 0& 1& 1& 1& 2 & 10\\
 \textbf{Aug} &                0& 0& 0& 3& 0& 0& 1& 0& 0& 1& 0& 2 & 7\\
 \textbf{Sep} &                0& 0& 0& 1& 1& 0& 0& 0& 0& 0& 1& 0 & 3 \\
 \textbf{Oct} &                1& 1& 0& 2& 0& 0& 1& 0& 0& 1& 1& 0 & 7\\
   \textbf{Nov} &              0& 1& 1& 1& 2& 0& 0& 2& 0& 1& 1& 0 & 9\\
 \textbf{Dec} &                0& 1& 1& 0& 0& 0& 1& 0& 0& 0& 0& 0 & 3\\ \hline
        \textbf{Total} & 13 & 4 & 7 & 10 & 8 & 4 & 5 & 3 & 4 & 9 & 7 & 8 & \textbf{82}
\end{tabular} 
\caption{Birth and deathday for Queen Victoria's descendants.} \label{tab: victoria}
\end{center}
\end{table}

Table \ref{tab: victoria} records the month of birth and death for $82$ descendants of Queen Victoria, occurs as an example in \cite{diaconisSturmfels}. The $\chi^2$ statistic is $115.6$ with $121$ degrees of freedom, which gives $p$-value $0.621$, suggesting we do not reject the null hypothesis for independence. The classical rules of thumb for validity of the chi-square approximation is that there is a minimum of $5$ entries per cell; this assumption is are badly violated in Table \ref{tab: victoria}, and there are too many tables with these margins for exact enumeration.

\begin{table}[] 
\begin{center}

\begin{tabular}{l|llll|l}
               & Black & Brown & Red & Blond & \textbf{Total} \\ \hline
Brown          & 68    & 119   & 26  & 7     & 220            \\
Blue           & 20    & 84    & 17  & 94    & 215            \\
Hazel          & 15    & 54    & 14  & 10    & 93             \\
Green          & 5     & 29    & 14  & 16    & 64             \\ \hline
\textbf{Total} & 108   & 286   & 71  & 127   & \textbf{592}  
\end{tabular}
\caption{Eye color vs.\ hair color for $n = 592$ individuals.} \label{tab: hair}
\end{center}
\end{table}

Table \ref{tab: hair} was analyzed in \cite{DEfron}; the $\chi^2$ statistic is $138.29$ with $9$ degrees of freedom.

\begin{figure}[ht]
\begin{center}
\includegraphics[scale=0.7]{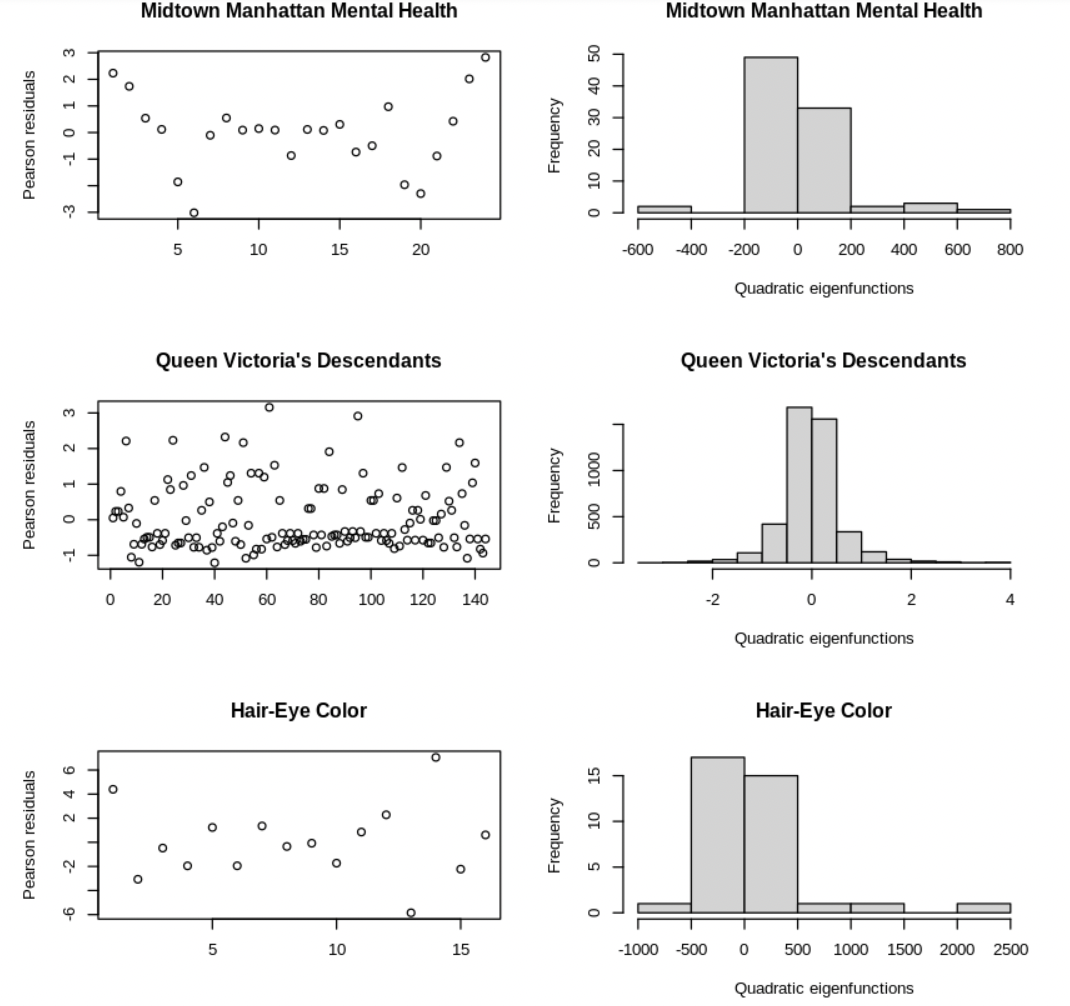}
\caption{The figures in the left column are the Pearson residuals, which are the normalized linear eigenfunctions; the $x$-axis indicates an arbitrary ordering of the residuals and the $y$-axis are the values of the residuals. The right column shows histograms for the normalized quadratic eigenfunctions.} \label{fig: stats}

\end{center}
\end{figure}

\paragraph{Residuals.} The linear eigenfunctions derived in Lemma \ref{lem: linearEF} are well known as `Pearson residuals' in the classical analysis of contingency tables \cite{agresti92}. The naturally scaled eigenfunctions
\[
\widehat{f}_{ij}(T) = \frac{T_{ij} - \lambda_i \mu_j/n}{\sqrt{\lambda_i \mu_j/n}}
\]
measure the departure of the table from the null model. Standard practice displays all of these in a two way array. Another scaling is inspired by the inner product space with respect to the Fisher-Yates distribution, so that the functions have norm $1$:
\[
\widetilde{f}_{ij}(T) =  \frac{T_{ij} - \lambda_i \mu_j/n}{\sqrt{c_{ij}}}, \quad c_{ij} := \frac{\lambda_i \mu_j(n - \lambda_i)(n - \mu_j)}{n^2(n-1)} = \E_{\pi_{\lambda, \mu}}\left[ (T_{ij} - \lambda_i \mu_j/n)^2 \right].
\]

Table \ref{tab: midtownPR} compares $\widehat{f}_{ij}$ and $\widehat{f}_{ij}$ for the Midtown Manhattan Mental Health dataset from Table \ref{tab: midtown}.

\begin{table}[!htb]

    \begin{minipage}{.5\linewidth}
      \centering
\begin{tabular}{c | c c c c }
& Well & Mild & Moderate & Impaired  \\ \hline
A & 2.233 & -0.104 & 0.114 & -1.965 \\
B &  1.737 & 0.546 & 0.078 & -2.298 \\
C &  0.538 & 0.090 & 0.305 & -0.885 \\
D &  0.117 & 0.148 & -0.737 & 0.423 \\
E & -1.858 & 0.092 & -0.498 & 2.018\\
F & -3.020 & -0.867 & 0.971 & 2.826
\end{tabular} 
      \caption{The standard Pearson residuals, $\widehat{f}_{ij}$.}
    \end{minipage}%
    \begin{minipage}{.5\linewidth}
      \centering

\begin{tabular}{c | c c c c }
& Well & Mild & Moderate & Impaired  \\ \hline
A & 2.695 & -0.142 & 0.141 & -2.446 \\
B &  2.083 &  0.741 &  0.096 & -2.844 \\
C &  0.656 & 0.124 & 0.379 & -1.111 \\
D &  0.147 &  0.211 & -0.950 & 0.551 \\
E & -2.245 & 0.125 & -0.615 & 2.515 \\
F & -3.587 & -1.165 &  1.177 &  3.462
\end{tabular} 
       \caption{The linear eigenfunctions $\widehat{f}_{ij}$ standardized to have norm $1$.}
    \end{minipage} 
    
        \caption{Comparison of the different normalizations $\widehat{f}_{ij}$ and $\widetilde{f}_{ij}$ of the linear eigenfunctions, for the dataset from Table \ref{tab: midtown}.} \label{tab: midtownPR}
\end{table}

\paragraph{Chi-squared Decomposition.} It is natural to try to use the quadratic eigenfunctions in a similar way. They do not have an interpretation as (observed - expected) for some observed statistic, but they do have expectation $0$ with respect to $\pi_{\lambda, \mu}$.



A potential use for the polynomial eigenfunctions is in a decomposition of the chi-squared statistic
\begin{align} \label{eqn: chiTest}
\chi^2(T) = \sum_{i, j} \frac{(T_{ij} - \lambda_i \mu_j/n)^2}{\lambda_i \mu_j/n}.
\end{align}
Under the null hypothesis that row and column features are independent, $\chi^2(T)$ has a limiting chi-squared distribution with $(I-1)\cdot(J-1)$ degrees of freedom. A drawback of using $\chi^2(T)$ is that when it is large enough to reject the null hypothesis, no additional information is given about \emph{why} the sample fails the test. This motivates the subject of partitioning the statistic into terms which could give insight into what parts of the table fail the independence hypothesis. The thesis \cite{salzman} contains a thorough review of this problem, and a theory for decomposing $\chi^2(T)$ using Markov chains on the space $\{1, \dots, I \} \times \{1, \dots, J \}$. Presented below is a natural decomposition using the linear and quadratic eigenfunctions $f_{ij}$ and $f_{(i,j), (i, j)}$.

Let $M_{ij} = \lambda_i \mu_j/n$, $K_{ij}, L_{ij}$ be such that
\begin{align*}
&K_{ij} = \frac{2 \lambda_i \mu_j - 2 \lambda_i - 2 \mu_j + n}{n - 2} \\
&L_{ij} = \frac{\lambda_i \mu_j(1 + \lambda_i \mu_j - \lambda_i - \mu_j)}{(n - 1)(n-2)},
\end{align*}
so that $f_{(i, j), (i, j)}(T) = T_{ij}^2 - K_{ij}T_{ij} + L_{ij}$. Then,
\begin{align*}
    \chi^2(T) &= \sum \frac{1}{M_{ij}}(T_{ij} - M_{ij})^2 = \sum \frac{1}{M_{ij}}(T_{ij}^2 - 2M_{ij}T_{ij} + M_{ij}^2) \\
    &= \sum \frac{1}{M_{ij}}(T_{ij}^2 - K_{ij}T_{ij} + L_{ij} + (2M_{ij} - K_{ij})T_{ij} + M_{ij}^2 - L_{ij}) \\
    &= \sum_{i, j} \frac{1}{M_{ij}}f_{(i, j), (i, j)}(T) + \sum_{i, j} \frac{2M_{ij} - K_{ij}}{M_{ij}}f_{(i, j)}(T) + \sum_{i, j} \frac{M_{ij}^2 - L_{ij}}{M_{ij}}.
\end{align*}


Figure \ref{fig: stats} shows histograms of the normalized quadratic eigenfunctions $f_{(i, j), (k, l)}/\E_\pi \left[ f_{(i, j), (k, l)} \right]$.

\subsection{Higher Dimensional Tables} \label{sec: multiway}


The random swap Markov chain can be used as inspiration for Markov chains on $m$-way contingency tables with fixed margins for $m > 2$. This section describes how that could go for $m = 3$.

Let $\lambda, \mu, \rho$ be three partitions of $n$ with $|\lambda| = I, |\mu| = J, |\rho| = K$. Let $\mathcal{T}_{\lambda, \mu, \rho}$ be the set of three-way tables with fixed margins determined by $\lambda, \mu, \rho$. That is $T = \{T_{ijk}: 1 \le i \le I, 1 \le j, \le J, 1 \le k \le K  \} \in \T_{\lambda, \mu, \rho}$ if
\begin{align*}
\sum_{jk} T_{ijk} = \lambda_i, \quad \sum_{ik} T_{ijk} = \mu_j, \quad \sum_{ij} T_{ijk} = \rho_k, \quad \text{for all} \,\,\, 1 \le i \le I, 1 \le j \le J, 1 \le k \le K. 
\end{align*}
The partitions $\lambda, \mu, \rho$ are the sufficient statistics for the complete independence model: The probability of an entry being $(i, j, k)$ is $p_{ijk} = \theta_i \eta_j \gamma_k$. A table can by thought of as tri-partite hypergraph, where each edge connects exactly three vertices. 

Representing a table by a set of $n$ tuples $(i, j, k)$, we can describe a similar adjacent swap Markov chain as: Pick $2$ tuples $(i_1, j_1, k_1), (i_2, j_2, k_2)$, pick $r \in \{1, 2, 3 \}$ uniformly, and swap the $r$th entry in the two tuples. For example,
\[
(i_1, j_1, k_1), (i_2, j_2, k_2) \to \begin{cases}
(i_2, j_1, k_1), (i_1, j_2, k_2) \\
(i_1, j_2, k_1), (i_2, j_1, k_2) \\
(i_1, j_1, k_2), (i_2, j_2, k_1)
\end{cases}.
\]
This move still corresponds to adding $\begin{pmatrix} - 1 & 1 \cr 1 & -1 \end{pmatrix}$ to a $2 \times 2$ submatrix of the contingency table. The probability of picking the two tuples $2T_{i_1j_1k_1} \cdot T_{i_2j_2k_2}/n^2$. To be precise about the Markov chain, let $F_{(i_1, j_1, k_1), (i_2, j_2, k_2), r}(T)$ denote the table where the swap was made in the indicated entries of the table at the $r$ index, with $r \in \{1, 2, 3 \}$. Then,
\begin{align*}
    P(T, F_{(i_1, j_1, k_1), (i_2, j_2, k_2), r}(T)) = \frac{1}{3} \cdot \frac{2 T_{i_1j_1k_1}T_{i_2j_2k_2}}{n^2}.
\end{align*}

These same moves were proposed in \cite{diaconisSturmfels} as input for a Metropolis Markov chain for log-linear models on multi-way tables (Section 4.2); for that Markov chain the $2$ tuples to be swapped are chosen uniformly from all entries (and if the swap move results in a negative value in the table, it is rejected), thus the chain has uniform stationary distribution. 

\begin{theorem}
The Markov chain $P$ on $\T_{\lambda, \mu, \rho}$ is connected and reversible with respect to the natural analog of Fisher-Yates for 3-way tables:
\[
\pi_{\lambda, \mu, \rho}(T) := \frac{1}{n!} \prod_{i, j, k} \frac{\lambda_i! \mu_j! \rho_k!}{T_{ijk}!}.
\] 
\end{theorem}

\begin{proof}
To see that the Markov chain is connected, we can define a partial order on the space $\T_{\lambda, \mu, \rho}$, analagous to the majorization order on $\T_{\lambda, \mu}$ used in Proposition \ref{prop: CTorder}: Write $T \prec T'$ if 
\[
\sum_{r = 1}^R \sum_{s = 1}^S \sum_{t = 1}^T T_{i_r j_s k_t} \le \sum_{r = 1}^R \sum_{s = 1}^S \sum_{t = 1}^T T'_{i_r j_s k_t}, \quad \text{for all} \,\,\, 1 \le R \le I, 1 \le s \le J, 1 \le T \le K.
\]
One move of the Markov chain corresponds to one edge in the lattice defined by this partial order (i.e.\ if $P(T, T') > 0$ then $T \prec T'$ or $T' \prec T$). Thus, every table can transition to the largest element and so the space is connected under $P$.

To see that the chain is reversible with the correct stationary distribution, given $T \in \T_{\lambda, \mu, \rho}$ suppose $T' = F_{(i_1, j_1, k_1), (i_2, j_2, k_2), r}(T)$ with $r = 1$ so that $P(T, T') > 0$. Then,
\begin{align*}
    \pi_{\lambda, \mu, \rho}(T') = \pi_{\lambda, \mu, \rho}(T) \cdot \frac{T_{i_1j_1k_1}T_{i_2j_2k_2}}{(T_{i_2j_1k_1} + 1)(T_{i1j_2k_2} + 1)}.
\end{align*}
Then,
\begin{align*}
    \pi_{\lambda, \mu, \rho}(T') P(T', T) &= \pi_{\lambda, \mu, \rho}(T) \cdot \frac{T_{i_1j_1k_1}T_{i_2j_2k_2}}{(T_{i_2j_1k_1} + 1)(T_{i1j_2k_2} + 1)} \cdot \frac{2 (T_{i_2j_1k_1} + 1)(T_{i1j_2k_2} + 1)}{3n^2} \\
    &=  \pi_{\lambda, \mu, \rho}(T) \cdot \frac{2T_{i_1j_1k_1}T_{i_2j_2k_2}}{3 n^2} = \pi_{\lambda, \mu, \rho}(T) P(T, T').
\end{align*}
The calculation is analogous for $r = 2, 3$.
\end{proof}

\begin{remark}
It would be interesting to have a double-coset representation for $\T_{\lambda, \mu, \rho}$, but we have not discovered one. 

\end{remark}

\subsection{Comparison of Markov Chains} \label{sec: metropolisCT}

This new chain on contingency tables can be compared to other Markov chains on contingency tables using simple spectral comparison techniques, as developed in \cite{diaconisSC}. 

The comparison technique relies on the variational characterization of the spectral gap of a Markov chain: If $P$ is a reversible Markov chain on $\Omega$ with stationary distribution $\pi$ and $\gamma = 1 - \lambda_2$ the spectral gap, then
\[
\gamma = \min_{f: \Omega \to \R: \text{Var}_\pi(f)} \frac{\mathcal{E}(f)}{\text{Var}_\pi(f)},
\]
where
\[
\mathcal{E}(f) := \frac{1}{2} \sum_{x, y \in \Omega} \left( f(x) - f(y) \right)^2 \pi(x) P(x, y).
\]
The \emph{relaxation time} $\tau$ is defined as the inverse of the absolute spectral gap $\tau = 1/\gamma$, and can be used as one measure of mixing. 

\begin{lemma}[Lemma 13.22 in \cite{levinPeres}]
Let $P$ and $\widetilde{P}$ be reversible transition matrices on $\Omega$ with stationary distributions $\pi$ and $\widetilde{\pi}$, respectively. If $\widetilde{\mathcal{E}}(f) \le \alpha \mathcal{E}(f)$ for all $f$, then
\begin{equation} \label{eqn: comparison}
\widetilde{\gamma} \le \left( \max_{x \in \Omega} \frac{\pi(x)}{\widetilde{\pi}(x)} \right) \alpha \gamma.
\end{equation}
\end{lemma}

The term $\max_{x \in \Omega} \pi(x)/\widetilde{\pi}(x)$ makes it difficult to compare Markov chains with stationary distributions which are very different. The uniform distribution and the Fisher-Yates distribution on $\T_{\lambda, \mu}$ are exponentially different and the ratio of the distribution cannot be bounded by a constant. 

\begin{example}
With $\lambda = (n - k, k), \mu = (n - \ell, \ell)$, with $k \le \ell \le n/2$, there are $k$ tables. The table with the smallest Fisher-Yates probability is
\[
T = \begin{pmatrix}
n - \ell - k & \ell \cr
k & 0
\end{pmatrix},
\]
with $\pi_{FY}(T) = (n - k)!(n- \ell)!/((n - k - \ell)! n!)$. Taking $k = \ell = n/2$ for example gives
\[
\pi_{FY}(T) = \frac{(n/2)!(n/2)!}{n!} \sim \frac{(n/2)^{n/2}(n/2)^{n/2}}{n^n} = \frac{1}{2^n},
\]
by Stirling's approximation. This is in sharp contrast with the uniform probability $\pi_U(T) = 1/k = 2/n$.
\end{example}

Instead, we can use the Metropolis algorithm on either the symmetric swap Markov chain or the random transpositions Markov chain to get two chains with the same distribution to compare. This is done in both cases in the following sections.

Let $P_{FY}$ be the random transpositions Markov chain on $\T_{\lambda, \mu}$ from Definition \ref{def: chain}. Let $P_U$ be the symmetric swap Markov chain which has stationary uniform distribution. That is,
\[
P_U(\mx, \my) = \begin{cases} \frac{2}{(IJ)^2} & \text{if} \quad \my = F_{(i_1, j_1), (i_2, j_2)}(\mx) \\
1 - \sum_{i_1 < i_2} \sum_{j_2 < j_2} \frac{2}{(IJ)^2} & \text{if} \quad \my = \mx \\
0 & \text{otherwise}
\end{cases}.
\]
Let $P_U^M$ be the chain from Metropolizing $P_{FY}$ to have uniform stationary distribution and $P_{FY}^M$ be the result of Metropolizing $P_U$ to have stationary distribution Fisher-Yates.

\begin{theorem} \label{thm: metropolisComparison}
Let $\tau_{FY}, \tau_{FY}^M, \tau_U, \tau_U^M$ be the associated relaxation times for the Markov chains $P_{FY}, P_{FY}^M, P_U, P_U^M$. Define
\begin{align*}
    &m_\lambda = \min \{ x_{ij} > 0 : \mx \in \T_{\lambda, \mu} \} \\
    &M_\lambda = \max \{ x_{ij} > 0 : \mx \in \T_{\lambda, \mu} \}
\end{align*}
If $|\lambda| = I, |\mu| = J$, then
\begin{enumerate}[(a)]
    \item 
\[
\frac{m_{\lambda, \mu}^2 (IJ)^2}{n^2} \tau_U^M  \le  \tau_U \le \frac{M_{\lambda, \mu}^2 (IJ)^2}{n^2} \tau_U^M,
\]

\item 
\[
 \frac{n^2}{(IJ)^2M_{\lambda, \mu}^4} \tau_{FY}^M \le \tau_{FY} \le \frac{n^2}{(IJ)^2m_{\lambda, \mu}^4} \tau_{FY}^M.
\]
\end{enumerate}
For any $\lambda, \mu$, note that $m_{\lambda, \mu} \ge 1$ and $M_{\lambda, \mu} \le \max(\lambda_i, \mu_j)$.
\end{theorem}


\begin{example}
If $\lambda = \mu = (c, c, \dots, c)$, $I = J = n/c$, then $m_{\lambda, \mu} = 1$ and $M_{\lambda, \mu} = c$. Then Theorem \ref{thm: metropolisComparison} shows
\[
\frac{n^2}{c^4} \tau_U^M \le \tau_U \le \frac{n^2}{c^2} \tau_U^M.
\]
The intuition is that if $n$ is growing and $c$ is fixed then the tables are relatively sparse, and so the Metropolis chain is much more likely than the uniform swap chain to propose moves which are accepted. Accepting more moves means exploring the state space more quickly, and thus achieving stationarity. Similarly, for the chains with Fisher-Yates stationary distribution in this setting,
\[
\frac{1}{n^2} \tau_{FY}^M \le \tau_{FY} \le \frac{c^4}{n^2} \tau_{FY}^M.
\]
Again, if $c$ is fixed and $n$ is larger the Markov chain utilizing the random transpositions has significantly smaller relaxation time than the Markov chain created by using the uniform swap moves with the Metropolis algorithm. Note that if $c = 1$, then the Fisher-Yates distribution is exactly the uniform distribution. 
\end{example}

The remainder of this section reviews the Metropolis-Hastings Algorithm, finds the transition probabilities $P_U^M$ and $P_{FY}^M$, and then proves Theorem \ref{thm: metropolisComparison}.

\paragraph{Metropolis-Hastings Algorithm} The random transpositions chain could be used as the proposal Markov chain in Metropolis-Hastings algorithm to sample from the uniform distribution. Suppose $P$ is a Markov chain on $\Omega$ and $\pi$ is a target distribution. A new Markov chain on $\Omega$ which is reversible with respect to $\pi$ is defined by: From $x$
\begin{enumerate}
    \item Generate a random candidate $y$ according to one step of the $P$ chain, i.e.\ $y \sim P(x, y)$.
    
    \item Compute the \emph{acceptance probability}
    \[
    A(x, y) = \min \left( 1, \frac{\pi(y) P(y, x)}{\pi(x) P(x, y)} \right).
    \]
    
    \item With probability $A(x, y)$, move to $y$. Otherwise, stay at the current state $x$.
\end{enumerate}

The transitions for the new Markov chain are defined
\[
K(x, y) = P(x, y) \cdot A(x, y).
\]

\paragraph{Uniform Distribution Comparison}
For two states $\mx, \my \in \T_{\lambda, \mu}$ such that $P(\mx, \my) > 0$, the acceptance probability can be computed. One way to see this is to note that since $P$ is reversible with respect to the Fisher-Yates distribution $\pi_{\lambda, \mu}$, it is
\[
\frac{P(\my, \mx)}{P(\mx, \my)} = \frac{\pi_{\lambda, \mu}(\mx)}{\pi_{\lambda, \mu}(\my)} = \frac{(x_{i_1, j_2} + 1)(x_{i_2, j_1} + 1)}{x_{i_1, j_1} \cdot x_{i_2, j_2}},
\]
supposing that $\my = F_{(i_1, j_1), (i_2, j_2)}(\mx)$. 

In conclusion, the Metropolis chain has transitions
\begin{align*}
P_U^M(\mx, \my) &= \frac{2x_{i_1, j_1} x_{i_2, j_2}}{n^2} \cdot \min \left( 1, \frac{(x_{i_1, j_2} + 1)(x_{i_2, j_1} + 1)}{x_{i_1, j_1} \cdot x_{i_2, j_2}} \right) =  \min \left( \frac{2x_{i_1, j_1} x_{i_2, j_2}}{n^2}, \frac{(2x_{i_1, j_2} + 1)(x_{i_2, j_1} + 1)}{n^2} \right) \\
&= \min \left( P_{FY}(\mx, \my), P_{FY}(\my, \mx) \right).
\end{align*}
This is clearly a symmetric Markov chain and has uniform stationary distribution. Note that this in general, any chain $P(x, y)$ when transformed via Metropolis to give the uniform stationary distribution is if the form $P_U(x, y) = \min(P(x, y), P(y, x))$.

To make the comparison of Dirichlet forms, since the stationary distributions are the same it is enough to compare the transition probabilities. This can be done,
\[
P_U(\mx, \my) = \frac{n^2}{m_{\lambda, \mu}^2 (IJ)^2} \cdot \frac{2 m_{\lambda, \mu}^2}{n^2} \le \frac{n^2}{m_{\lambda, \mu}^2 (IJ)^2} P_U^M(\mx, \my),
\]
and similarly
\[
P_{FY}(\mx, \my) = \frac{n^2}{M_{\lambda, \mu}^2 (IJ)^2} \cdot \frac{2 M_{\lambda, \mu}^2}{n^2} \ge \frac{n^2}{M_{\lambda, \mu}^2 (IJ)^2} P_U^M(\mx, \my).
\]
This gives the result, since
\begin{align*}
    \alpha_1 P_U^M(\mx, \my) \le P_U(\mx, \my) \le \alpha_2 P_U^M(\mx, \my) \implies \frac{1}{\alpha_2} \tau_U^M \le \tau_U \le \frac{1}{\alpha_1} \tau_U^M.
\end{align*}

\paragraph{Fisher-Yates Distribution Comparison}

In the other direction, we could take the symmetric swap Markov chain and Metropolize it to have stationary distribution Fisher-Yates. The acceptance probability in this direction is
\begin{align*}
A(\mx, \my) &= \min \left(1, \frac{\pi_{FY}(\my) P_U(\my, \mx)}{\pi_{FY}(\mx) P_U(\mx, \my)} \right) \\
&= \min \left(1, \frac{\pi_{FY}(\my)}{\pi_{FY}(\mx) } \right) = \min \left(1, \frac{x_{i_1, j_1} \cdot x_{i_2, j_2}}{(x_{i_1, j_2} + 1)(x_{i_2, j_1} + 1)} \right).
\end{align*}
The transition probabilities, for $\my = F_{(i_1, j_1), (i_2, j_2)}(\mx)$ are then
\begin{align*}
    P_{FY}^M(\mx, \my) = \frac{2}{(IJ)^2} \min \left(1, \frac{x_{i_1, j_1} \cdot x_{i_2, j_2}}{(x_{i_1, j_2} + 1)(x_{i_2, j_1} + 1)} \right).
\end{align*}
To compare this with the random transpositions chain $P_{FY}$,
\begin{align*}
    P_{FY}(\mx, \my) &= \frac{2 x_{i_1,j_1} x_{i_2, j_2}}{n^2} \le \frac{2 M_{\lambda,\mu}^2}{n^2} \\
   & =\frac{(IJ)^2M_{\lambda, \mu}^4}{n^2} \cdot \frac{2 m_{\lambda, \mu}^2}{M_{\lambda, \mu}^2(IJ)^2} \le \frac{(IJ)^2M_{\lambda, \mu}^4}{n^2} \cdot  P_{FY}^M(\mx, \my)
\end{align*}
And similarly
\begin{align*}
P_U(\mx, \my) &= \frac{2 x_{i_1,j_1} x_{i_2, j_2}}{n^2} \ge \frac{2 m_{\lambda,\mu}^2}{n^2} \\
   & =\frac{(IJ)^2m_{\lambda, \mu}^4}{n^2} \cdot \frac{2 M_{\lambda, \mu}^2}{m_{\lambda, \mu}^2(IJ)^2} \ge \frac{(IJ)^2m_{\lambda, \mu}^4}{n^2}  \cdot  P_{FY}^M(\mx, \my)
\end{align*}

This proves part (b) of Theorem \ref{thm: metropolisComparison}.

\subsection{Parabolic Subgroups of $GL_n$} \label{sec: GLnparabolic}

Contingency tables also arise from double cosets using the parabolic subgroups of $GL_n(q)$, as described in \cite{karpThomas}.

\begin{definition}
Let $\alpha = (\alpha_1, \dots, \alpha_k)$ be a partition of $n$. The parabolic subgroup $P_\alpha \subset GL_n(q)$ consists of all invertible block upper-triangular matrices with diagonal block sizes $\alpha_1, \dots, \alpha_k$
\end{definition}

Section 4 of \cite{karpThomas} shows that if $\alpha, \beta$ are two partitions of $n$, the double cosets $P_\alpha \backslash GL_n(q) / P_\beta$ are indexed by contingency tables with row sums $\alpha$ and column sums $\beta$. Proposition 4.37 contains the size of a double-coset which corresponds to the table $T$, with $\theta = 1/q$
\begin{equation} \label{eqn: parabolicGLn}
\theta^{- n^2 + \sum_{1 \le i < i' \le I, 1 \le j < j' \le J} T_{ij} T_{i'j'}}(1 - \theta)^n \cdot \frac{\prod_{i = 1}^I[\alpha_i]_\theta! \prod_{j = 1}^J [\beta_j]_\theta!}{\prod_{i, j} [M_{ij}]_\theta!}, \quad I = |\alpha|, J = |\beta|.
\end{equation}
Dividing \eqref{eqn: parabolicGLn} by $(1-\theta)^n$ and setting $\theta = 1$ recovers the usual Fisher-Yates distribution for partitions $\alpha, \beta$. It remains an open problem to investigate these probability distributions and Markov chains on the double cosets of $GL_n(q)$ from parabolic subgroups.

\bibliographystyle{abbrv}
\bibliography{citations}

\end{document}